\documentclass[a4paper]{amsart}

\usepackage{amssymb,amscd}
\usepackage{amsmath,amsfonts,latexsym}
\usepackage[dvips]{graphicx}
\usepackage{color}
\usepackage[english]{babel}
\usepackage[latin1]{inputenc}
\usepackage{enumerate,afterpage}
\usepackage{mathtools}
\usepackage{tabu}
\usepackage{array,booktabs}

\usepackage{pst-node}
\usepackage{tikz-cd}
\usetikzlibrary{matrix,arrows}
\usepackage[all]{xy}

\usepackage{mathrsfs}
\usepackage{mathtools}
\usepackage{stmaryrd}

\usepackage{amssymb,amscd}
\usepackage{amsmath,amsfonts,latexsym}
\usepackage{graphicx}
\usepackage{color}
\usepackage[english]{babel}
\usepackage[latin1]{inputenc}
\usepackage{enumerate,afterpage}
\usepackage{tabu}
\usepackage[all]{xy}
\usepackage{faktor}

\makeatletter
\DeclareRobustCommand*{\mfaktor}[3][]
{
   { \mathpalette{\mfaktor@impl@}{{#1}{#2}{#3}} }
}
\newcommand*{\mfaktor@impl@}[2]{\mfaktor@impl#1#2}
\newcommand*{\mfaktor@impl}[4]{
   \settoheight{\faktor@zaehlerhoehe}{\ensuremath{#1#2{#3}}}%
   \settoheight{\faktor@nennerhoehe}{\ensuremath{#1#2{#4}}}%
      \raisebox{-0.5\faktor@zaehlerhoehe}{\ensuremath{#1#2{#3}}}%
      \mkern-4mu\diagdown\mkern-5mu%
      \raisebox{0.5\faktor@nennerhoehe}{\ensuremath{#1#2{#4}}}%
}
\makeatother

\newcounter{notes}%

\newtheorem{cor}{Corollary}[section]

\newtheorem{prop}[cor]{Proposition}

\newtheorem{Theorem}[cor]{Theorem}

\newtheorem{Proposition}[cor]{Proposition}
\newtheorem{Corollary}[cor]{Corollary}

\newtheorem{introthm}{Theorem}

\usepackage{amsthm}

\makeatletter
\newtheorem*{rep@theorem}{\rep@title}
\newcommand{\newreptheorem}[2]{%
\newenvironment{rep#1}[1]{%
 \def\rep@title{#2 \ref{##1}}%
 \begin{rep@theorem}}%
 {\end{rep@theorem}}}
\makeatother

\newreptheorem{theorem}{Theorem}

\theoremstyle{definition}

 \newtheorem{Definition}[cor]{Definition}

 \newtheorem{Example}[cor]{Example}

\newtheorem{Remark}[cor]{Remark}
\newtheorem{remark}[cor]{Remark}
\theoremstyle{plain}

\def\co{\colon\thinspace}

\newcommand{\bR}{{\mathbf R}}
\newcommand{\bC}{{\mathbf C}}
\newcommand{\bQ}{{\mathbf Q}}
\newcommand{\bH}{{\mathbf H}}
\newcommand{\bZ}{{\mathbf Z}}
\newcommand{\bN}{{\mathbf N}}

\newcommand{\HH}{{\mathbb H}}

\newcommand{\N}{{\mathbb N}}

\newcommand{\Z}{{\mathbb Z}}

\newcommand{\SU}{\mathsf{SU}}

\newcommand{\Sp}{\mathsf{Sp}}
\newcommand{\PSL}{\mathsf{PSL}}

\newcommand{\U}{\mathsf{U}}

\newcommand{\SO}{\mathsf{SO}}
\newcommand{\GL}{\mathsf{GL}}

\newcommand{\bpm}{\begin{pmatrix}}
\newcommand{\epm}{\end{pmatrix}}

\newcommand{\XG}{\mathcal X_{\sf G}} 
\newcommand{\Heis}{{\rm Heis}}

\newcommand{\ax}{{\rm Axis}}

\let\oldtocsection=\tocsection

\let\oldtocsubsection=\tocsubsection

\let\oldtocsubsubsection=\tocsubsubsection

\renewcommand{\tocsection}[2]{\hspace{0em}\oldtocsection{#1}{#2}}
\renewcommand{\tocsubsection}[2]{\hspace{1em}\oldtocsubsection{#1}{#2}}
\renewcommand{\tocsubsubsection}[2]{\hspace{2em}\oldtocsubsubsection{#1}{#2}}

\AtBeginDocument{%
   \def\MR#1{}
}

\usepackage{mathtools}

\DeclarePairedDelimiter\floor{\lfloor}{\rfloor}

\makeatletter
\newcommand{\para}[1]{%
  \par
  \addvspace{\medskipamount}
  \textbf{#1\@addpunct{.}}\enspace\ignorespaces
}
\makeatother

\begin{document}

\title[Geometric Limits]{Geometric limits of cyclic subgroups of $\SO_0(1, k+1)$ and $\SU(1, k+1)$}

\author[S. Maloni]{Sara Maloni}
\address{SM: Department of Mathematics, University of Virginia}
\email{sm4cw@virginia.edu}
\urladdr{https://sites.google.com/view/sara-maloni}

\author[M.B. Pozzetti]{Maria Beatrice Pozzetti}
\address{MBP: Department of mathematics, University of Heidelberg}
\email{pozzetti@mathi.uni-heidelberg.de}
\urladdr{www.mathi.uni-heidelberg.de/$\sim$pozzetti}

\thanks{Maloni was partially supported by grant DMS-1506920, DMS-1650811, DMS-1839968 and DMS-1848346 from the National Science Foundation. Pozzetti acknowledges funding by the Deutsche Forschungsgemeinschaft (DFG, German Research Foundation)--427903332.
The authors also acknowledge support from U.S. National Science Foundation grants DMS-1107452, 1107263, 1107367 ``RNMS: GEometric structures And Representation varieties'' (the GEAR Network).}

\date{\today}

\begin{abstract}
We study geometric limits of convex-cocompact cyclic subgroups of the rank 1 groups $\SO_0(1, k+1)$ and $\SU(1, k+1)$. We construct examples of sequences of subgroups of such groups $\sf G$ that converge algebraically and whose geometric limit strictly contains the algebraic limit, thus generalizing the example first described by J{\o}rgensen for subgroups of $\SO_0(1,3)$. We also give necessary and sufficient conditions for a subgroup of $\SO_0(1, k+1)$ to arise as geometric limit of a sequence of cyclic subgroups. We then discuss generalizations of such examples to sequence of representations of free groups, and applications of our constructions in that setting.
\end{abstract}

\maketitle

\tableofcontents

\section{Introduction}\label{intro}
\addtocontents{toc}{\protect\setcounter{tocdepth}{1}}
Given a finitely generated torsion free group $\Gamma$, and a  Lie group $\sf G$, we can consider the representation variety $\mathcal{R} = \mathcal{R}(\Gamma, \sf G) := \mathrm{Hom}(\Gamma, \sf G)$, which is the set of homomorphisms from $\Gamma$ into $\sf G$. In this variety, an important subset is $\mathcal{DF} = \mathcal{DF}(\Gamma, \sf G)$, the subset of faithful representations with discrete image. The set $\mathcal{DF}$, as a subset of $\mathcal{R}$, inherits its \textit{algebraic topology}, that is the topology of pointwise convergence. On the other hand, $\mathcal{DF}$ can also be considered as a subset of the set of closed subgroups of $\sf G$ and so it can be endowed with the Chabauty (or compact-open) topology. This induces the \textit{geometric topology} on $\mathcal{DF}$.  The relationship between these two topologies has been investigated for many years, above all in the case of ${\sf G} = \PSL(2, \bC)$. In this article we extend some of the results known for representations into $\PSL(2, \bC)$ to representations into a more general Lie group $G$ of rank one, such $\SO_0(1, k+1)$ and $\SU(1, k+1)$.
 
 In the following discussion, we will assume the sequence $\left(\rho_n\right)_{n \in \bN}$ in $\mathcal{DF}$ converges algebraically to $\rho_\infty$ and $\left(\rho_n(\Gamma)\right)_{n \in \bN}$ converges geometrically to $\Gamma_G$. In this case one easily sees that $\rho_\infty(\Gamma)$ is a subgroup of $\Gamma_G$. If $\rho_\infty(\Gamma) = \Gamma_G$, we say that the convergence is strong. J{\o}rgensen \cite{jor_onc} was the first to show examples of sequences of representations from $\Gamma = \bZ$ into ${\sf G} = \SO_0(1, 3)$ such that $\rho_\infty(\Gamma)$ is a proper subgroup of $\Gamma_G$. 
We aim at investigating to what extent this theory survives for Lie groups different than $\SO_0(1,3)$. In this article we will concentrate on the two Lie groups of rank one that have a more concrete and workable model: $\SO_0(1, k+1)$ and $\SU(1, k+1)$. We plan to discuss analogue results for general Lie groups, including those of higher rank, in an upcoming paper \cite{MP2}, which will use much more heavily a Lie theoretic setup. In \cite{MP2} we will be specifically interested in understanding limits of Anosov representations, which can be considered as generalization of convex-cocompact representations when the target group has higher real rank. Convex-cocompact cyclic representations $\rho\co\bZ\to \sf G$ for rank one groups $\sf G$ like $\SO_0(1, k+1)$ and $\SU(1, k+1)$ corresponds to representations such that $\rho(1)$ is loxodromic, and for Anosov representations each element is, in particular, loxodromic. This is why for many of the results here we restrict to loxodromic elements, an assumption that considerably simplifies the discussion. On the other hand all of the results presented here can be discussed for the more general setting of discrete and faithful representations admitting parabolics. We will point out through the paper what are the correct generalizations to that more general setting. See, for example, Remark \ref{r.parabolic1}.

The first result of this paper is a generalization of J\o rgensen's work.
\begin{introthm}\label{t.INTRO1}
	Let ${\sf G_1} = \SO_0(1, k+1)$ and ${\sf G_2} = \SU(1, k+1)$. For any $l_1, l_2 \in \bN$ such that $1 \leq l_1 \leq \floor*{\frac{k}{2}}$ and $1\leq l_2\leq {k}$, there exist sequences $\left(\rho_n\co\bZ\to{ \sf G}_i\right)_{n \in \bN}$ of convex-cocompact representations such that the algebraic limit is a discrete and faithful representation $\rho_{\infty}\co\bZ\to { \sf G}_i$ and such that the geometric limit $\Gamma_G$ of the subgroups $\rho_n(\bZ)$ is isomorphic to $\bZ^{l_i+1}$ for $i = 1, 2$. 
\end{introthm}

Observe that the largest rank of a discrete abelian subgroup of $\SU(1, k+1)$ is $k+1$, and, since the geometric limit of a sequence of abelian groups is necessarily abelian, one cannot expect better results than the ones obtained in Theorem \ref{t.INTRO1}. On the other hand, $\SO_0(1, k+1)$ admits discrete subgroups of rank $k$, which is bigger than  $\floor*{\frac{k}{2}}+1$ as soon as $k\geq 3$. It is thus natural to wonder if the bound of Theorem \ref{t.INTRO1} is sharp. In the second result of the paper we prove that this is the case.
\begin{introthm}\label{it.2a}
Let $(\rho_n\co\bZ\to \SO_0(1,k+1))_{n\in\bN}$ be a sequence of convex-cocompact representations converging geometrically to a discrete group $\Gamma_G$. Then the rank of $\Gamma_G$ is at most $\floor*{\frac{k}{2}}+1$.
\end{introthm}

 As a third result we prove that the bound on $l$ provided by Theorem \ref{t.INTRO1} is also a sufficient condition on a subgroup of $\SO_0(1, k+1)$ which guarantees that it arises as geometric limit of cyclic subgroups.
\begin{introthm}\label{it.2}
Let $l=\floor*{\frac{k}{2}}$ and let $\Delta<\SO_0(1, k+1)$ be a discrete, torsion free, abelian subgroup  of rank at most $l+1$. Then there is a sequence of convex cocompact representations $(\rho_n:\bZ\to \SO_0(1,k+1))_{n\in\bN}$ converging algebraically and geometrically and such that the geometric limit $\Gamma_G$ is equal to $\Delta$. 
\end{introthm}
 
Understanding geometric limits of cyclic subgroups is a key step in understanding geometric limits of general convex cocompact subgroups. Indeed, on the one hand, the setting is simple enough to allow for a full classifications of the possible limits, such as the one provided in Theorem \ref{it.2}, on the other hand it is often the case that if a sequence $\rho_n$ of representations of the group $\Gamma$ has a geometric limit that strictly contains the algebraic limit, then it is possible to find a cyclic subgroup $\langle \gamma\rangle<\Gamma$ for which the same occurs. In order to illustrate this phenomena, we construct, in the last section of the paper examples of sequences of representations of free groups whose restriction to a suitably chosen cyclic subgroup gives the examples constructed in Theorem \ref{t.INTRO1}.

\begin{introthm}\label{it.3}
Let ${\sf G}$ be either $\SO_0(1, k+1)$ or $\SU(1, k+1)$. There exists a sequence $(\rho_n\co F_2 = \langle a, b\rangle\to {\sf G})_{n\in\bN}$ of convex-cocompact representations  such that the restriction $\rho_n|_{\langle a\rangle}$ is the sequence of representations constructed in Theorem \ref{t.INTRO1}. In these examples, $\rho_n$ converges algebraically to $\rho_\infty$ and geometrically to a discrete subgroup $\Gamma_G$ isomorphic to $\bZ^2\ast \bZ$ and properly containing $\rho_\infty( F_2)$.
\end{introthm}

One key tool we develop in the paper is a precise understanding of the interplay between the Jordan and Iwasawa decomposition of loxodromic elements of the group $\sf G$. In particular we define, in Section \ref{sec_gen_jor}, the notion of \emph{geometric data} for such an element $g$, which corresponds to a geometric description of its Jordan decomposition. Namely, it is enough to record the position of its two fixed points, its translation length and its rotation angles in order to be able to reconstruct a loxodromic element up to conjugacy in $K$. In addition, from the geometric data for an element $g$ we can easily compute the analogue data associated to any power $g^n$ for any $n \in \bN$. On the other hand, we can relate this to the Iwasawa decomposition of the element in question, which allows us to write, in each case, the matrix corresponding to $g$ with respect to a fixed basis. This is key in understanding which sequences of elements converge, which in turn allows us to compute the geometric limit of sequences of representations. In particular, given a sequence $(\rho_n\co\bZ\to {\sf G})_{n \in \bN}$, we obtain a precise criterion for a sequence of elements to converge algebraically (Proposition \ref{p.converge} and \ref{p.convergeC}), which we then use to construct various examples of sequences converging algebraically and not strongly.

J\o rgensen's example, as well as its generalizations, had numerous applications in Kleinian theory, leading to a better understanding of the topology of the set of convex cocompact representations and its closure in the character variety. In particular, they helped illustrating phenomena such as \emph self-bumping, and were also used to show that the action of the mapping class group doesn't extend continuously to many classical compactifications of the set of convex cocompact representations or slices thereof. In addition, they were the basis of various compactness arguments in three dimensional geometry and topology (see Section \ref{sec:Jor} for some examples and references). Despite many such applications also use deep results so far only known for hyperbolic $3$--manifolds,  we hope that new results in this direction will be made possible in the future by the fast growing understanding of the geometry of the locally symmetric spaces associated to general Anosov groups.

\subsection*{Outline of the paper}
 In Section \ref{back} we discuss the background. We recall important facts about the symmetric spaces associated to a group $\sf G$ of real rank one, the Chabauty topology for closed subgroups of $\sf G$, the geometric and the algebraic topology on spaces of representations into $\sf G$ and their relationship. In Section \ref{sec_gen_jor} we  describe some generalizations of the J{\o}rgensen's example  to the Lie groups  $\SO_0(1, k+1)$ and $\SU(1, k+1)$, while developing the abstract setup needed in the next sections. These are the sequences mentioned in Theorem \ref{t.INTRO1}. In Section \ref{thats_all} we  show that these are the best results one can get in the $\SO_0(1, k+1)$-case, proving Theorems \ref{it.2a} and \ref{it.2}. Finally, in Section \ref{s.free} we extend these results to  free groups, proving Theorem \ref{it.3} and concluding the proof of Theorem \ref{t.INTRO1}.

\addtocontents{toc}{\protect\setcounter{tocdepth}{2}}

\section{Background}\label{back}

\subsection{Symmetric spaces of rank one} \label{rank_one}
In this section we will revise some background material about Riemannian symmetric spaces of real rank $1$ tailored to a reader with some familiarity with the hyperbolic three dimensional space. In particular we will describe a upper-half space model which will be very important in the rest of paper. A standard reference for this material is \cite{chen_greenberg}.

The Riemannian symmetric spaces of real rank $1$ and of negative curvature are the hyperbolic spaces $\HH^{k+1}_{\mathbf{F}}$, where $\mathbf{F}$ is the algebra:
\begin{enumerate}[(a)]
  \item $\mathbf{F} = \mathbf{R}$ of real numbers;
  \item $\mathbf{F} = \mathbf{C}$ of complex numbers;
  \item $\mathbf{F} = \mathbf{H}$ of quaternion numbers;
  \item $\mathbf{F} = \mathbf{O}$ of octonion numbers (only in the case $k = 1$).
\end{enumerate}

These spaces correspond to the homogeneous spaces $\faktor{\sf G}{\sf K}$, where $\sf G$ is a semisimple Lie group $\sf G$ of real rank $1$ and $\sf K$ is its maximal compact subgroup. In particular we can write these spaces as follows:
\begin{enumerate}[(a)]
  \item $\HH^{k+1}_{\mathbf{R}} = \faktor{\SO_0(1, {k+1})}{\SO(k)}$;
  \item $\HH^{k+1}_{\mathbf{C}} = \faktor{\SU(1, {k+1})}{\mathsf{U}(k)}$;
  \item $\HH^{k+1}_{\mathbf{H}} = \faktor{\Sp(1, {k+1})}{\Sp(k)}$;
  \item $\HH^2_{\mathbf{O}} = \faktor{F_4}{\mathsf{Spin}(9)}$.
\end{enumerate}
Note that in our descriptions the group $\sf G$ doesn't act faithfully on $\XG$ since its (finite) center acts trivially. However, for the sake of the computation we decided to priviledge the groups $\sf G$ having a simpler matrix expression.

In these notes we will  focus our attention to the real and complex cases. As a result, from now on, we will let $\mathbf{F} \in \{\mathbf{R}, \mathbf{C}\}$, even if most results of the next section work in the case $\mathbf F=\bH$ as well. 

\subsubsection{The projective model}

Let $V = V^{1, {k+1}}(\mathbf{F})$ be the vector space $\mathbf{F}^{k+2}$ endowed with the  $\mathbf{F}$--Hermitian form $\Psi$ defined by: 
$$\Psi(\mathbf{z}, \mathbf{w}) = \overline{z_0} w_{k+1} + \overline{z_{k+1}} w_0 - \sum_{j = 1}^{k} \overline{z_j} w_j,$$
for all $\mathbf{z} = (z_0, \ldots, z_{k+1}), \mathbf{w}= (w_0, \ldots, w_{k+1}) \in V$. Here $\overline{z_i}$ denotes the standard conjugation in $\mathbf{F}$. 

An $\mathbf{F}$--linear transformation $g \in \mathrm{SL}(V) \cong \mathrm{SL}(k+2,\mathbf{F})$ is $\Psi$--\textit{unitary} if it preserves $\Psi$, that is, if $\Psi(\mathbf{z}, \mathbf{w}) = \Psi(g(\mathbf{z}), g(\mathbf{w}))$ for all $\mathbf{z}, \mathbf{w} \in V$. The set of $\Psi$--unitary transformations defines the group $\SU(1, {k+1}; \mathbf{F})$, which is the automorphism group of $V$:
$$\SU(1, {k+1}; \mathbf{F}) = \left\{g \in \mathrm{SL}(V) \mid \Psi(\mathbf{z}, \mathbf{w}) = \Psi(g(\mathbf{z}), g(\mathbf{w})) \quad \forall \mathbf{z}, \mathbf{w} \in V\right\}.$$   
We denote by $V^+$ the cone of positive vectors $V^+ = \{z \in V \mid \Psi(z, z) > 0\}$ and by $V^0$ the cone of null vectors $V^0 = \{z \in V \mid \Psi( z, z) = 0\}$. Note that these subsets are invariant under $\SU(1, k+1; \mathbf{F})$.

Let $\mathbb{P}(V)$ be the projective space of $V$. The \textit{projective model} of the hyperbolic space $\HH^{k+1}_{\mathbf{F}}$ is 
$$\HH^{k+1}_{\mathbf{F}} := \mathbb{P} (V^+) \subset \mathbb{P}(V).$$ 
It is easy to check that $\SU(1, {k+1}; \mathbf{F})$ acts transitively on $\HH^{k+1}_{\mathbf{F}}$ leaving it invariant.

The space $\HH^{k+1}_{\mathbf{F}}$ is  a model for the symmetric space, as it can be identified with the coset space $\faktor{\SU(1, {k+1}; \mathbf{F})}{\mathsf{S}(\mathsf{U}(1; \mathbf{F}) \times \mathsf{U}({k+1}; \mathbf{F}))}$. Indeed the stabilizer of any point in the action is isomorphic to the subgroup of the group $\mathsf{U}(1; \mathbf{F}) \times \mathsf{U}(k+1; \mathbf{F})$ consisting of matrices of determinant one as we will now see. Let $\mathcal{E} = \{e_0, \ldots, e_{k+1}\}$ be the standard basis of $V=\mathbf{F}^{k+2} $. Consider the point $\mathbf q :=[e_0+e_{k+1}] \in \HH^{k+1}_{\mathbf{F}}$. In order to compute the stabilizer of $\mathbf q$ in $\SU(1, k+1; \mathbf{F})$, note that the orthogonal complement of $\mathbf q$ is the subspace $W = \langle e_1, \ldots, e_{k}, \frac{1}{\sqrt 2}(e_0-e_{k+1}) \rangle$ which is isomorphic to $\mathbf{F}^{k+1}$. This gives  an orthonormal basis for  the restriction of $\Psi$ to $W$, which is positive definite, and is therefore the standard Hermitian form. As a result the group of $\mathbf F$-unitary transformations preserving $W$ is isomorphic to $\mathsf{U}({k+1}; \mathbf{F})$. If $g\in\SU(1, {k+1}; \mathbf{F})$ has the property that $g(e_0+e_{k+1}) = a(e_0+e_{k+1})$, then $|a| = 1$ and $g$ leaves $W$ invariant. So, with respect to the basis $\left\{\frac{1}{\sqrt2}(e_0+e_{k+1}), e_1,\ldots,e_{k},\frac{1}{\sqrt2}(e_0-e_{k+1})\right\}$, the element $g$  is of the form 
$ g = \begin{bmatrix}
    a & 0  \\
     0 & A 
 \end{bmatrix}$, 
where $a =\det(A)^{-1}$ and $A \in \mathsf{U}({k+1}; \mathbf{F})$, as we wanted.

\subsubsection{The boundary and classification of isometries}\label{s.isom}
Let $\overline{\HH^{k+1}_{\mathbf{F}}}$ denote the closure of $\HH^{k+1}_{\mathbf{F}}$ in $\mathbb{P}(V)$. The boundary  $\partial \HH^{k+1}_{\mathbf{F}}$ of $\HH^{k+1}_{\mathbf{F}}$ in $\mathbb{P}(V)$ is the projectivization $\mathbb{P}(V^0)$ of the null cone $V^0$. The group $\SU(1, {k+1}; \mathbf{F})$ acts  transitively on pairs of distinct points in $\partial \HH^{k+1}_{\mathbf{F}}$. Any element  $g \in \SU(1, {k+1}; \mathbf{F})$ leaves $\overline{\HH^{k+1}_{\mathbf{F}}}$ invariant and since $\overline{\HH^{k+1}_{\mathbf{F}}}$ is a closed ball, $g$ has a fixed point in $\overline{\HH^{k+1}_{\mathbf{F}}}$. In particular, we say that an element $g \in \SU(1, {k+1}; \mathbf{F}) \setminus \{\mathrm{Id}\}$ is
\begin{itemize}
  \item \textit{Elliptic} if it has at least one fixed point in $\HH^{k+1}_{\mathbf{F}}$.
  \item \textit{Parabolic} if it has exactly one fixed point in $\partial\HH^{k+1}_{\mathbf{F}}$. Furthermore we say that it is \textit{unipotent} if all its eigenvalues are equal to 1.
  \item \textit{Loxodromic} if it has no fixed points in $\HH^{k+1}_{\mathbf{F}}$, and exactly two fixed points in $\partial\HH^{k+1}_{\mathbf{F}}$. Furthermore we say that it is \textit{hyperbolic}  if it has positive real eigenvalues.
\end{itemize}
One can check that:
\begin{itemize}
  \item Any elliptic element has eigenvalues of norm one, and are conjugate to elements in $\mathsf{S}(\mathsf{U}(1; \mathbf{F}) \times \mathsf{U}(k+1; \mathbf{F}))$. (See discussion above).
    \item Any parabolic element $g$ has $k$ eigenvalues of norm one, and it  admits a unique decomposition as $g = pe$, where $p$ is unipotent,  $e$ is elliptic, and $p$ and $e$ commute. 
  \item Any loxodromic element $g$ has $k$ eigenvalues of norm one and it has a unique decomposition as $g = he$, where $h$ is hyperbolic, $e$ is elliptic, and $h$ and $e$ commute. Any loxodromic element which fixes  the two points $[e_0], [e_{k+1}] \in \partial\HH^{k+1}_{\mathbf{F}}$ has the form 
$$g = \begin{bmatrix}
 e^{y}\lambda &  & \\
   & A &\\
& & e^{-y} {\lambda}
  \end{bmatrix},$$ where $y \in \mathbf{R}$, $\lambda \in \mathbf{F}$ with $|\lambda| = 1$, and $A \in \mathsf{U}(k; \mathbf{F})$ with $\mathrm{det}(A) = \lambda^{-2}$.
\end{itemize}

In order to study non-elliptic elements it is often also useful to consider the upper-half space model; this is obtained from the projective model by taking an affine chart in which exactly one point in $\partial \HH^{k+1}_{\mathbf{F}}$ is at infinity.

\subsubsection{Upper-half space model}

We consider the affine chart on the complement of the hyperplane $\{x_0=0\}$ given by $(u_1,\ldots, u_{k+1})\mapsto [1,u_1,\ldots, u_{k+1}]$. The \textit{upper half-space domain} $\mathscr{S}$ is the preimage of $\HH^{k+1}_{\mathbf{F}}$ in this chart, and can be explicitely described as 
$$\mathscr{S} = \mathscr{S}^{k+1}(\mathbf{F}) = \left\{u \in \mathbf{F}^{k+1} \mid \Re (\hat{u}_{k+1}) > \frac{1}{2}\sum_{j = 1}^{k} |\hat{u}_j|^2 \right\}.$$
It is sometimes called \textit{Siegel domain}.  We denote by $\infty:= [e_0] =\partial\HH^{k+1}_{\mathbf{F}} \setminus \partial \mathscr{S}$ the only point in $\partial \HH^{k+1}_{\mathbf{F}}$ outside of our chosen affine chart. Its stabilizer in $\SU(1, {k+1}; \mathbf{F})$ is
$${\sf G}_{\infty} = \{g \in {\SU}(1, {k+1}; \mathbf{F}) \mid g(\infty) = \infty\}.$$
Any element $g \in {\sf G}_\infty$ can be written, with respect to the standard basis as a matrix 
$$g =  \begin{bmatrix}
\lambda  & \lambda \overline{a}^T A & s\\
0 & A & a\\
0 & 0 & \overline{\lambda}^{-1}
\end{bmatrix},\quad \text{ where }\quad \left\{\begin{array}{l}\lambda \in \mathbf{F}^*\\ a\in \mathbf{F}^{k}\\A \in \U(k; \mathbf{F})\\\mathrm{det}(A)\lambda\overline{\lambda}^{-1}=1\\\Re (\lambda^{-1} s) = \frac{1}{2}|a|^2 = \frac{1}{2} \overline{a}^T a\end{array}\right. .$$
Here $|a|$ is the Frobenius or Euclidean norm, which is defined as the square root of the sum of the squares of the absolute value of its elements.

\subsubsection{A model for the boundary}\label{s.bdry}
It will be important for us to have a good understanding of the action of elements in ${\sf G}_\infty$ on the boundary $\partial\HH^{k+1}_{\mathbf{F}}$, as well as to have a convenient model for $\partial\HH^{k+1}_{\mathbf{F}}\setminus \{\infty\}$.

For this observe that ${\sf G}_\infty$, being a parabolic group, admits a \emph{Levi decomposition} ${\sf G}_\infty=N\rtimes L$ where $N$ is its unipotent radical, and $L$ is its Levi factor. In our specific case of interest $L$ is the subgroup of block diagonal matrices
$$ L = \left\{ 
 \begin{bmatrix} 
\lambda  & 0& 0\\
 0 & A & 0\\
 0 & 0 & \overline{\lambda}^{-1}
 \end{bmatrix} \left|
 \begin{array}{l} 
  \lambda \in \mathbf F^{*}, \\A\in \U(k;\mathbf F),\\ \mathrm{det}(A)\lambda\overline{\lambda}^{-1}=1
  \end{array}\right.\right\}.$$
Observe that the group $L$ consists  of the elements in $\sf G$ that preserve the geodesic with endpoints $\infty:= [e_0]$ and $0:=[e_{k+1}]$. The unipotent radical $N$ of ${\sf G}_\infty$ consists of unipotent matrices of the form
$$N = N_{\mathbf{F}} = \left\{ 
  \begin{bmatrix}
1  &  \overline{a}^T  & s\\
0 & E_{k} & a\\
0 & 0 & 1
\end{bmatrix}\left|
 \begin{array}{l} 
  a \in \mathbf F^{k}, \\
s \in \mathbf{F},\\
\Re ( s) = \frac{1}{2}|a|^2
  \end{array}\right.\right\}.$$
In particular, if $\mathbf F=\mathbf R$, the number $s$ is uniquely determined by $a$, and the unipotent group $N_{\mathbf{R}}$ can be identified with $\mathbf R^{k}$ with its usual group structure, while, if $\mathbf F=\bC$, the imaginary part of $s$ can be chosen freely and  thus the group $N_{\mathbf{C}}$ identifies with the Heisenberg group
$$\Heis_{k}:=\bC^{k}\rtimes \bR$$ 
with group structure
$$(a_1,b_1)\cdot (a_2,b_2)=(a_1+a_2, b_1+b_2+\Im(\overline{a_1}^Ta_2).$$
In the following, with a slight abuse of notation, we will often identify the groups $\bR^{k}$ (resp. $\Heis_{k}$) with the group $N$, thus leaving implicit the isomorphisms 
$$\begin{array}{rccc}
\psi_{\mathbf{R}}\co&\bR^{k}&\to&N_{\mathbf{R}}\\
&a&\mapsto& \begin{bmatrix}
1  &  \overline{a}^T  & \frac{1}{2}|a|^2\\
0 & E_{k} & a\\
0 & 0 & 1
\end{bmatrix}
\end{array}
\quad\quad
\begin{array}{cccc}
\psi_{\mathbf{C}}\co&\Heis_{k}&\to& N_{\mathbf{C}}\\
&(a,b)&\mapsto&  \begin{bmatrix}
1  &  \overline{a}^T  & \frac{1}{2}|a|^2+ib\\
0 & E_{k} & a\\
0 & 0 & 1
\end{bmatrix}
\end{array}
$$
The group $N$ acts simply transitively on $\partial\HH^{k+1}_{\mathbf{F}}\setminus \{\infty\}$, and we can thus use it to give a parametrization of such set, based at the point $0=[e_{k+1}]$:
\begin{Proposition}\label{l.param}
The maps
$$\begin{array}{cccc}
\phi_{\mathbf{R}}\co&\bR^{k}&\to& \partial\HH^{k+1}_{\mathbf{R}}\setminus \{\infty\}\\
&a&\mapsto& \begin{bmatrix} \frac{1}{2}|a|^2\\a\\1\end{bmatrix}
\end{array}
\quad\quad
\begin{array}{cccc}
\phi_{\mathbf{C}}\co&\Heis_{k}&\to& \partial\HH^{k+1}_{\mathbf{C}}\setminus \{\infty\}\\
&(a,b)&\mapsto& \begin{bmatrix} \frac{1}{2}|a|^2+ib\\a\\1\end{bmatrix}
\end{array}
$$
give a parametrization of $\partial\HH^{k+1}_{\mathbf{F}}\setminus \{\infty\}$. This is equivariant with the $N_{\mathbf{F}}$--actions described above.
\end{Proposition}
Furthermore, since the action of $N$ on $\partial\HH^{k+1}_{\mathbf{F}}\setminus \{\infty\}$ is simply transitive, we have the following:
\begin{Proposition}\label{l.convtoU}
Let $(g_n)_{n\in\bN}$ be a sequence of elements in $\SU(1,{k+1};\mathbf F)$ converging to an element $g_\infty$. Assume that $g_\infty$ belongs to $N$. Then
$$g_\infty=\psi\circ\phi^{-1}\lim_{n \to \infty}(g_n\cdot [e_{k+1}]).$$
\end{Proposition} 

\subsection{Geometric and algebraic convergence}\label{geom-alg} 

In this paper we will consider the space ${\rm Hom}(\Gamma,\sf G)$ of representations $\rho\co\Gamma \to {\sf G}$, where $\Gamma$ is a fixed abstract group and ${\sf G}$ is a semisimple Lie group. We consider two types of convergence: the algebraic and the geometric convergence.

The algebraic convergence is the pointwise convergence, or equivalently the convergence induced from the topology of the representation variety ${\rm Hom}(\Gamma, \sf G)$.
\begin{Definition}[Algebraic convergence]
  A sequence $(\rho_n\co\Gamma \to {\sf G})_{n \in \bN}$ \textit{converges algebraically} to $\rho_\infty\co\Gamma \to {\sf G}$ if for all $\gamma \in \Gamma$, $\{\rho_n(\gamma)\}_n$ converges to $\rho_\infty(\gamma)$ (in the topology of ${\sf G}$). The representation $\rho_\infty$ is called the \textit{algebraic limit} of $(\rho_n)_{n \in \bN}$.
\end{Definition} 

In order to have the language to discuss geometric convergence, we need a short excursus about Chabauty topology. A good reference about it is the paper by Abert, Bergeron, Biringer, Gelander, Nikolov, Rimbault and Samet \cite{7samurai}.

\begin{Definition}\label{chabauty}
Given a locally compact, second countable group $\sf G$, let $\mathrm{Sub}_{\sf G}$ denote the set of closed subgroups of $\sf G$. The \textit{Chabauty (or compact-open) topology} \cite{chabauty} on $\mathrm{Sub}_{\sf {\sf G}}$ is generated by open sets of the form
\begin{enumerate}
  \item $\mathcal{O}_1(K) = \{H \in \mathrm{Sub}_{\sf G} \mid H \cap K = \emptyset\}$ for all $K \subset {\sf G}$ compact;
  \item $\mathcal{O}_2(U) = \{H \in \mathrm{Sub}_{\sf G} \mid H \cap U \neq \emptyset\}$ for all $U \subset {\sf G}$ open.
\end{enumerate}
Alternatively, a sequence $\{H_n\}_{n \in \bN}$ in $\mathrm{Sub}_{\sf G}$ converges to $H \in \mathrm{Sub}_{\sf G}$ if and only if:
\begin{enumerate}
  \item For every $h \in H$, there exists a sequence $\{h_n\}_{n \in \N}$ in ${\sf G}$ such that $h_n \in H_n$ and $h_n \to h$ in ${\sf G}$;
  \item For every subsequence $\{h_{n_i}\}$ such that $h_{n_i} \in H_{n_i}$ and $h_{n_i} \to \hat{h}$, then $\hat{h} \in H$.
\end{enumerate}
\end{Definition}
This leads to the notion of geometric convergence.

\begin{Definition}[Geometric convergence]
 A sequence $(\Gamma_n)_{n \in \bN}$ of closed subgroups of ${\sf G}$ \textit{converges geometrically} to $\Gamma_G$ if it converges in the Chabauty topology, see Definition \ref{chabauty}. One can see Remark \ref{r.GH} for the reason behind the name \emph{geometric}.
\end{Definition}
For convenience, if we have a sequence of representations $(\rho_n\co\Gamma \to {\sf G})_{n \in \bN}$ such that the associated subgroups $(\rho_n(\Gamma))$ converge geometrically to $\Gamma_G$, we will often say, with a slight abuse of notation, that $(\rho_n)_{n \in \bN}$ converges geometrically to $\Gamma_G$.

Convergence of discrete groups of semisimple Lie groups in the Chabauty topology can be understood more geometrically via Gromov-Hausdorff convergence of the associated pointed orbifolds.
\begin{Remark}[On Gromov-Hausdorff convergence]\label{r.GH}
We denote by $\XG = \faktor{\sf G}{\sf K}$ the  Riemannian symmetric space associated to a semisimple Lie group $\sf G$. Observe that the choice of the presentation $\XG = \faktor{\sf G}{K}$ implicitly includes the choice of a basepoint $[\sf K]$. A sequence of discrete subgroups $H_n<\sf G$ converges to a discrete subgroup $H$ in the Chabauty topology if and only if the pointed quotient orbifolds $\left(\mfaktor{H_n}{\XG}, [K]\right)$ converge to $\left(\mfaktor{H}{\XG}, [K]\right)$ in the pointed Gromov-Hausdorff topology.  As we won't need Gromov--Hausdorff convergence  in the paper, we refer to Canary, Epstein and Green \cite{can_not} for details. 
\end{Remark}

There are some relations between algebraic and geometric convergence. This result is a direct consequence of the definitions:
\begin{Proposition}\label{subset}
  If $(\rho_n\co\Gamma \to \sf G)_{n \in \bN}$ converges algebraically to $\rho_\infty$ and $(\rho_n(\Gamma))_{n \in \bN}$ converges geometrically to $\Gamma_G$, then $\rho_\infty(\Gamma)\subseteq\Gamma_G$.
\end{Proposition}

When the last inclusion is an equality, we have `strong convergence': 
\begin{Definition}[Strong convergence]
  A sequence $(\rho_n\co\Gamma \to {\sf G})_{n \in \bN}$ \textit{converges strongly} to $\rho_\infty\co\Gamma \to {\sf G}$ if it converges algebraically to $\rho_\infty$ and $(\rho_n(\Gamma))_{n \in \bN}$ converges geometrically to $\rho_\infty(\Gamma)$.
\end{Definition}

Direct properties of the Chabauty topology have useful applications in the study of geometric limits. For example it is well known that the Chabauty topology is compact, separable and metrizable, see \cite{benedetti-petronio}. Furthermore, clearly, any Chabauty limit of abelian subgroups is abelian.
 In addition the following fact is easy to check from the definition. We say that a sequence in $\mathrm{Sub}_{\sf G}$ that is \emph{uniformly discrete}, if there exists an open neighbourhood $U$ of the identity $\mathrm{Id}$ in ${\sf G}$ such that $H_n \cap U = \{\mathrm{Id}\}$ for all $n$.

\begin{Proposition}\label{l.udisc}
 Let ${\sf G}$ be a connected Lie group and let $(H_n)_{n\in \bN}$ be a sequence in $\mathrm{Sub}_{\sf G}$ that is \emph{uniformly discrete}. If $(H_n)_{n\in \bN}$ converges toward a group $H$ in $\mathrm{Sub}_{\sf G}$ in the Chabauty topology, then $H$ is discrete. Under the same assumption, if all the groups $H_n$ are torsion free, so is $H$.
\end{Proposition}
We immediately deduce the following result.
\begin{Corollary}\label{disc_Z}
  Given an algebraically converging sequence $(\rho_n\co\bZ \to \sf G)_{n\in\bN}$ of uniformly discrete and faithful representations, then the geometric limit $\Gamma_G$  is discrete and it is abelian (and so contained in a maximal torus of $\sf G$).
\end{Corollary}
The assumption of the existence of $U$ is crucial for both statements of Proposition \ref{l.udisc}. We construct,  in Section \ref{s.counter}, examples of discrete torsion-free subgroups whose limit is non-discrete and  has torsion. Observe that in these examples the sequences of groups cannot be uniformly discrete.

If, instead, the group is non-radical, for example a non-abelian free group, or a non-elementary hyperbolic group, we don't need to assume that the representations are uniformly discrete to deduce discreteness of the limit. Recall that a group is called \textit{non-radical} if it does not contain infinite normal nilponent groups.
\begin{Proposition}[{\cite[Proposition 8.9]{Kap}}]\label{prop.2.14}
  Let $\Gamma$ be any non-radical group and $(\rho_n\co\Gamma \to {\sf G})_{n \in \bN}$ be a sequence of discrete representations converging algebraically to $\rho_\infty$ and geometrically to $\Gamma_G$. Then $\Gamma_G$ is a discrete subgroup of $\sf G$. In particular, $\rho_\infty(\Gamma)$ is discrete.
\end{Proposition}

We conclude the section discussing two additional results that follow directly from the definition of geometric limit, but will be crucial in our study of generalized J\o rgensen examples.

\begin{Proposition}\label{conjug}
	Let $p \in \XG =\faktor{\sf G}{\sf K}$ and $(\Gamma_n)_{n \in \bN}$ be a sequence of subgroups of $\sf G$ which converges geometrically to the subgroup $\Gamma_G$. Let $(k_n)_{n \in \bN}$ be a sequence of elements in $\mathrm{Stab}_p(\sf G) \cong \sf K$ such that $\lim_n k_n = k_\infty$. Then we have
$$k_n \Gamma_n k_n^{-1} \xrightarrow{\text{geom.}} k_\infty \Gamma_G k_\infty^{-1}.$$
\end{Proposition}
We directly deduce:

\begin{Proposition}\label{l.dense}
The set of subgroups $\Delta$ that can be obtained as geometric limit of sequences of cyclic subgroup is closed in ${\rm Sub}(\sf G)$, and is invariant under conjugation in $G$. 
\end{Proposition}
\begin{proof}
The first property follows directly as the Chabauty topology is metrizable \cite{Bir}. The second is a consequence of the fact that if $\Delta$ is the geometric limit of the sequence $\rho_n(\bZ)$, then $g\Delta g^{-1}$ is the limit of the sequence $\rho_n^g(\bZ)$ defined by $\rho_n^g(1)=g\rho_n(1)g^{-1}$.
\end{proof}

\subsection{Geometric and algebraic convergence in $\SO_0(1,3)$} \label{sec:Jor}

We conclude the background section discussing J{\o}rgensen examples following Thurston \cite{thurston-notes}, as well as a description of some of the multiple applications of geometric convergence to the theory of Kleinian groups. 

J{\o}rgensen \cite{jor_onc} studied the sequence $(\rho_n\co\bZ \to \mathrm{Isom}^+(\HH_{\bR}^{3}) \cong {\sf PSL}(2, \bC))_{n\in\bN}$ of convex-cocompact representations defined by:
$$\rho_n(1) = g_n := \begin{bmatrix}
    \exp(\omega_n) & n \mathrm{sinh}(\omega_n)  \\
    0 & \exp(-\omega_n) 
\end{bmatrix},$$ 
 where $\omega_n = \frac{1}{n^2}+i\frac{\pi}{n}$. It is immediate to verify that the axis $\ax_n = \ax(\rho_n(1))$ of $g_n$ is the geodesic between $a_n = -\frac{n}{2}\in \bC$ and $\infty$. For any point $x\in \HH_{\bR}^3$ and for any $n\in \bN$, the element $\rho_n(1)$ moves $x$ around the cone $C_n(x)$ with axis $\ax_n$ and containing $x$. As $n$ goes to infinity, the distance between $\ax_n$ and $x$ diverges  and the surfaces $C_n$ limit to the horosphere in $\HH_{\bR}^3$ based at $\infty$ containing $x$. The elements $\rho_n(1)$ are chosen so that $\rho_n(n)$ translates $x$ along the Euclidean line in $\partial C_n(x)$ joining $x$ and $a_n$,  and so that the two elements $\rho_n(1)$ and $\rho_n(n)$ move $x$ by roughly the same amount. This is because the elements $\rho_n(1)$ were chosen to have translational part $\frac{1}{n^2}$ along the axis $\ax_n$ and rotational part $\frac{\pi}{n}$. We can then see that:
\begin{itemize}
  \item $\left(\rho_n\right)_{n \in \bN}$ converges algebraically to $\rho_{\infty}\co\bZ\to {\sf PSL}(2, \bC)$ defined by $\rho_{\infty}(1) =  \begin{bmatrix}
      1 & \pi i \\
      0 & 1 
  \end{bmatrix}  $; and
  \item $\left(\rho_n(\bZ)\right)_{n \in \bN}$ converges geometrically to $\Gamma_G = \left\langle \begin{bmatrix}
      1 & \pi i \\
      0 & 1
  \end{bmatrix},     \begin{bmatrix}
          1 & 1 \\
          0 & 1
      \end{bmatrix} \right\rangle$.
\end{itemize}

Thurston \cite{thu_hyp2}, building on J{\o}rgensen's work, discussed examples of representations with similar features, but where $\Gamma = \pi_1(\Sigma)$, and Kerckhoff--Thurston \cite{ker_non}, Anderson-Canary \cite{and_alg} and Brock \cite{bro_ite2, bro_ite} discussed other examples, each one with its own new `exotic' phenomena. Another important result in this direction was proven by Anderson and Canary \cite{and_cor} who used three-dimensional hyperbolic geometry to show that if the algebraic limit of a sequence of representations in $\PSL(2, \bC)$ does not contain any parabolic element, then the geometric and algebraic limit coincide. 

All these examples had multiple applications in Kleinian theory. In fact Kerckhoff--Thurston \cite{ker_non} example  was used to show that the action of the mapping class group on the Bers compactification of Teichm\"uller space is non-continuous.  Anderson and Canary's example \cite{and_alg} showed that the (marked) homeomorphism type does not necessarily vary continuously over the space 
$$\mathrm{AH}(M) = \faktor{\mathcal{DF}(\pi_1(M), \PSL(2,\bC))}{\PSL(2,\bC)}$$
 of all (marked) hyperbolic $3$--manifolds homotopy equivalent to a fixed compact, orientable, hyperbolizable $3$--manifold $M$ with boundary. Brock's example was key in proving that ending laminations do not vary continuously in $\mathrm{AH}(M)$, see \cite{bro_bou}. 

Another application of J\o rgensen example is to questions related to the topology of the space $\mathrm{AH}(M)$. The interior of $\mathrm{AH}(M)$ is well-understood thanks to the work Ahlfors, Bers, Kra, Marden, Maskit, Sullivan, Thurston and others, but the topology of $\mathrm{AH}(M)$, induced by the algebraic convergence mentioned above, is quite complicated and is not well understood in most cases. For example, Anderson and Canary \cite{and_alg}  showed that the connected components of its interior can \textit{bump}, that is, they can have intersecting closures. In the case of $3$--manifolds with incompressible boundary, Anderson, Canary and McCullough \cite{and_the} characterized exactly which components can bump. For $M = \Sigma_g \times [0,1]$  the interior of $\mathrm{AH}(M)$, the so-called quasi-Fuchsian space, is connected, but McMullen \cite{mcmullen:complex}  showed that it \textit{self-bumps}, that is that there are points $p$ in the boundary of $\mathrm{AH}(M)$ such that the intersection of the interior of $\mathrm{AH}(M)$ with sufficiently small neighbourhoods of $p$ is disconnected. Bromberg and Holt \cite{bro_sel} showed that self-bumping happens any time $M$ contains a primitive essential annulus.  Bromberg \cite{bro_the} and Magid \cite{mag_def}  showed that  $\mathrm{AH}(M)$ is not even locally connected. The study of the relations between geometric and algebraic convergence has also lead, through active research in the last years, to exclude self bumping at many boundary points, see \cite{bro_loc, bro_loc2} and references therein for more details. All of these examples come from studying sequences of representations converging algebraically, but not strongly. 

\section{Generalized J{\o}rgensen examples in $\SO_0(1,k+1)$ and $\SU(1,k+1)$}\label{sec_gen_jor}
\addtocontents{toc}{\protect\setcounter{tocdepth}{1}}

We consider sequences $(\rho_n\co\bZ\to \sf G)_{n\in\bN}$ of discrete and faithful representations algebraically converging to $\rho_\infty\co\bZ\to \sf G$, where $\sf G$ is a real semisimple Lie group of real rank one. In particular, we will analyze here only the cases ${\sf G} =\SO_0(1, k+1)$, and $\sf G= \SU(1, k+1)$ for $k\geq 2$,  and  write the proof in a language and notation that can be generalized to other semisimple Lie groups, since we plan to discuss that in an upcoming work \cite{MP2}. We will treat the two groups in two separate sections, but following a similar strategy.

The two subsections are structured as follows:  we determine the geometric data of the elements satisfying the aforementioned assumptions (Definition \ref{d.gdata1} and \ref{d.gdata2}); we compute the corresponding algebraic data, namely the matrices representing the elements with respect to the standard basis (Proposition \ref{p.algdata1} and \ref{p.algdata2}); we give necessary and sufficient conditions on the geometric data guaranteeing that a sequence converges (Proposition \ref{p.converge} and \ref{p.convergeC}), and we construct J{\o}rgensen examples `of maximal rank' (Propositions \ref{hyp_Jor} and \ref{hyp_Jor2}). Note that the proof of Theorem \ref{t.INTRO1} will be concluded in Section \ref{s.free} where we will  show that these examples are discrete.

\begin{remark}
In this section, as well as in the whole paper, we will only study loxodromic elements, since the application we have in mind are to convex cocompact (or more generally Anosov) representations, for which each element is loxodromic. It is however possible to do a parallel study of parabolic elements obtaining similar results with minor modifications. See also Remark \ref{parab_rk}.
\end{remark}

\subsection{Real hyperbolic space}\label{real_jor}

Let ${\sf G} = \SO_0(1, k+1)$ be the group of isometries of the real hyperbolic space $\HH^{k+1}_{\mathbf{R}}$, which we can identify with the Riemannian symmetric space $\HH^{k+1}_{\mathbf{R}} = \faktor{\SO_0(1, k+1)}{\SO(k+1)}$. 

We first proceed to define the geometric data associated to a loxodromic element $g\in\SO_0(1,k+1)$. We use the upper half-space model for $\HH^{k+1}_\bR$ discussed in Section \ref{rank_one}, where we consider the quadratic form $\Psi$ of signature $(1, k+1)$ defined by the matrix $Q = \begin{bmatrix}
 &  & 1\\
 & -E_{k} &\\
1 & &  
\end{bmatrix}$. We define $$\SO_0(1, k+1) = \{A \in \GL(k+2, \bR) \mid A^T Q A = Q, \mathrm{det}(A) = 1\}.$$ The compact centralizer of the geodesic in  $\HH_\bR^{k+1}$  with endpoints $0$ and $\infty$ is given by
$$\SO(k) \cong  \left\{B= \begin{bmatrix}
1 &  & \\
 & A &\\
 & &  1
\end{bmatrix} 
\in \GL(k+2, \bR) \left|\; \begin{array}{l}A^T A =  E_{k},\\ \mathrm{det}(A) = 1\end{array}\right.\right\} \subset \SO_0(1, k+1).$$  
The stabilizer ${\sf G}_\infty$ of $\infty$ is the model for the parabolic subgroups of $\sf G$. We already know that 
$$ {\sf G}_\infty = \mathrm{Stab}_G(\infty) = \left\{ 
 B = \begin{bmatrix} 
 \lambda  & \lambda a^T A &  \frac{\lambda |a|^2}{2}\\
 0 & A & a\\
 0 & 0 & \lambda^{-1}
 \end{bmatrix} \left|
 \begin{array}{l} 
  \lambda \in \bR^*, \\ a\in \bR^k, \\A\in \SO(k)
  \end{array}\right.\right\}.$$ 
Observe that $\lambda = \pm e^y$ for $y \in \bR$, but it will be enough for our purposes to treat the case in which $\lambda$ is positive. If $y \neq 0$, then $B$ represents a loxodromic element with two fixed points and with translation length $|y|$. Every loxodromic element in ${\sf G}_\infty$ can be uniquely determined by:
\begin{itemize}
  \item its second fixed point $\mathbf{x} \in \bR^k\cong\partial\HH^k_\bR\setminus\{\infty\}$;
  \item its hyperbolic translation part $y \in \bR$;
  \item its elliptic part, which corresponds to a matrix $A\in \SO(k)$.
\end{itemize}  
In addition, the matrix $A \in \SO(k)$ can be diagonalised, that is, up to choosing a suitable orthonormal basis for $\bR^k$, we can represent $A$ as follows:
\begin{equation}\label{e.A}A = A_{\theta_1, \ldots, \theta_l} = \begin{pmatrix} 
     \begin{matrix}
       \cos \theta_1  & \sin \theta_1 \\
        -\sin \theta_1 & \cos \theta_1
     \end{matrix}
     &  &  \\
     & \ddots & \\
   & &
     \begin{matrix}
        \cos \theta_l  & \sin \theta_l \\
         -\sin \theta_l & \cos \theta_l
      \end{matrix}
\end{pmatrix} \text{ if } k = 2l\in 2\mathbf{N}\end{equation} 
or 
\begin{equation}\nonumber
A  = A_{\theta_1, \ldots, \theta_l} = \begin{pmatrix} 
     \begin{matrix}
       \cos \theta_1  & \sin \theta_1 \\
        -\sin \theta_1 & \cos \theta_1
     \end{matrix}
     &  &  &\\
     & \ddots & & \\
   & &
     \begin{matrix}
        \cos \theta_l  & \sin \theta_l \\
         -\sin \theta_l & \cos \theta_l
      \end{matrix} &\\
           &  & &1\\
\end{pmatrix} \text{ if } k = 2l+1\in 2\mathbf{N}+1.
\end{equation}

Hence we have the following definitions:
\begin{Definition}[Well positioned $g\in\SO_0(1,k+1)$]
Let $g\in\SO_0(1,k+1)$ be a loxodromic element with Iwasawa decomposition $g=h_ge_g$.  We say that \emph{$g$ is well positioned}, if the attractive fixed point of $g$ is $\infty$ and that the elliptic element is block diagonal in the standard basis. 
\end{Definition}
\begin{Definition}[Geometric Data for well positioned $g\in\SO_0(1,k+1)$]\label{d.gdata1}
The \emph{geometric data} associated to $g$ is the $(k+l+1)$--tuple
$$\left(\mathbf{x}, y, \theta_1, \ldots, \theta_l\right) \in \bR^k \times \bR_+ \times  (\bR/2\pi \bZ)^l,$$
where 
\begin{itemize}
\item $\mathbf{x}$ are the coordinates of the repulsive fixed point of $g$, 
\item $y$ is the translation length of $h_g$, 
\item $\theta_1, \ldots, \theta_l$ are the rotation angles of $e_g$.
\end{itemize}

\end{Definition}
\begin{Remark}
For any loxodromic element $g$ it is possible to find $k\in\sf K$ such that $g^k=kgk^{-1}$ is well positioned. Furthermore there exists a unique such conjugate $g^k$ with the additional property that $x_{2i+1}=0$ for all $i\in \{1,\ldots,l\}$, $x_{2}\leq x_{4}\ldots\leq x_{2l}$ and $\theta_i\leq  \theta_{i+1}$ if $x_{2i}=x_{2i+2}$.
\end{Remark}

The next proposition shows how to compute the algebraic data corresponding to a given geometric data for $g$, that is, the matrix representing the element $g$.

\begin{prop}\label{p.algdata1}
If the well positioned loxodromic element $g\in\SO_0(1,k+1)$ has geometric data $\left(\mathbf{x}, y, \theta_1, \ldots, \theta_l\right)\in \bR^k \times \bR_+ \times  (\bR/2\pi \bZ)^l$, then the matrix $M_{(\mathbf{x}, y, \theta_1, \ldots, \theta_l)}$ representing $g$ is given by 
$$M_{(\mathbf{x}, y, \theta_1, \ldots, \theta_l)} = \begin{pmatrix} 
        \lambda  & \lambda \mathbf{v}^T A & \frac{\lambda |\mathbf{v}|^2}{2}\\
        0 & A & \mathbf{v}\\
        0 & 0 & \lambda^{-1}
        \end{pmatrix},$$
where 
\begin{itemize}
\item $A=A_{\theta_1, \ldots, \theta_l}$ is the matrix in Equation \eqref{e.A},
\item $\lambda=e^{y} \in (1, \infty)$,
\item $\mathbf{v} = -A \mathbf{x} +\lambda^{-1} \mathbf{x}\in \bR^k.$
\end{itemize}
\end{prop}

\begin{proof}
The element $g$  is conjugated to an element in the stabilizer $L$ of the pair $(0,\infty)$ via the unique unipotent element $u_{\bf x}$ fixing $\infty$ and such that $u_{\bf x}\cdot 0= {\bf x}$. In particular,
$$u_{\bf x}=\begin{bmatrix} 
    1  & \mathbf{x}^T  & \frac{1}{2}|\mathbf{x}|^2\\
    0 & 1 & \mathbf{x}\\
    0 & 0 & 1
    \end{bmatrix} .$$
It is easy to compute that $L$ has the following form:
$$ L = P \cap \mathrm{Stab}_G(\mathbf{0}) = \left\{ 
 B = \begin{bmatrix} 
 \lambda  & 0 & 0\\
 0 & A & 0\\
 0 & 0 & \lambda^{-1}
 \end{bmatrix} \left|
  \begin{array}{l} 
   \lambda \in \bR_{+},\\ A\in \SO(k)
   \end{array}\right.\right\}.$$
As a result, the matrix $M_{(\mathbf{x}, y,\theta_1, \ldots, \theta_l)}$ associated with the geometric data $(\mathbf{x}, y,\theta_1, \ldots, \theta_l) $, is given by:
 \begin{equation}
  M_{(\mathbf{x}, y, \theta_1, \ldots, \theta_l)} =  \begin{bmatrix} 
    1  & \mathbf{x}^T  & \frac{1}{2}|\mathbf{x}|^2\\
    0 & 1 & \mathbf{x}\\
    0 & 0 & 1
    \end{bmatrix}
 \begin{bmatrix} 
     \lambda  & 0  & 0\\
     0 & A & 0\\
     0 & 0 & \lambda^{-1}
     \end{bmatrix} 
\begin{bmatrix} 
       1  & -\mathbf{x}^T  & \frac{1}{2}|\mathbf{x}|^2\\
       0 & 1 & -\mathbf{x}\\
       0 & 0 & 1
       \end{bmatrix} =
 \begin{bmatrix} 
        \lambda  & \lambda \mathbf{v}^T A & \frac{\lambda |\mathbf{v}|^2}{2}\\
        0 & A & \mathbf{v}\\
        0 & 0 & \lambda^{-1}
        \end{bmatrix},
 \end{equation}
 where $$\mathbf{v} = -A \mathbf{x} +\lambda^{-1} \mathbf{x}\in \bR^k.$$ 
\end{proof}

With this at hand, it is easy to give necessary and sufficient conditions on the geometric data of a sequence of well positioned loxodromic elements $g_n \in {\sf G}_\infty$ guaranteeing that the sequence converges in $\SO_0(1,k+1)$.
\begin{prop}\label{p.converge}
Let $(g_m)_{m\in\N}\subset {\sf G}_\infty\subset\SO_0(1,k+1)$ be a sequence of well positioned loxodromic elements with geometric data $$\left(\mathbf{x}_m, y_m, \theta_{1,m}, \ldots, \theta_{l,m}\right)\in \bR^k \times \bR_+ \times   (\bR/2\pi \bZ)^l.$$
Assume, furthermore, that
\begin{itemize}
\item$x_{2i-1, m}= 0$ for all $i = 1,\ldots, l$
\item  $x_{2,m}\leq x_{4,m}\ldots\leq x_{2l,m}$ 
\item $x_{2l,m}\to \infty$.
\end{itemize} Then the sequence converges algebraically if and only if   
\begin{enumerate}[(a)]
\item $\theta_{i,m}x_{2i,m}$ converges for all $i = 1, \ldots, l$; 
\item $y_m x_{2l,m}$ converges.
\end{enumerate}
In  case also $x_{2,m}$ diverges, then the algebraic limit is the unipotent matrix associated to the vector $\bf v$ with coordinates 
$$\left\{\begin{array}{rl}
v_{2j-1}&= -\displaystyle{\lim_m \theta_{j,m}x_{2j, m}}\\
v_{2j}&= -\displaystyle{\lim_m y_mx_{2j,m}}\\
v_{2l+1}&= 0.
\end{array}\right.$$
\end{prop}

\begin{proof}
In order to understand the conditions under which the elements $g_m$ converge as $m$ goes to infinity, we need to compute approximations of the elements $\mathbf{v}_m$ determining the matrix $M_m$ associated, via Proposition \ref{p.algdata1}, to the geometric data of $g_m$. To this aim, observe that, with our choice of basis of $\bR^k$, the vectors $\mathbf{v}_m $ are given by 
\begin{equation}\label{v_m_real}
	\mathbf{v}_m = -A_m \mathbf{x}_m  + \lambda_m^{-1}\mathbf{x}_m.
\end{equation} 

We now prove the first  statement which will also allow us to compute the suitable Taylor expansions to verify the second statement.
Since, by assumption, the last coordinate $x_{2l,m}$ diverges, the norm of the vector $\mathbf{x}_m$ diverges as $m$ goes to infinity. Since $\mathbf{v}_m$ is the difference of two vectors of norm $\|A_m \mathbf{x}_m\| = \|\mathbf{x}_m\|$ and $\|\mathbf{x}_m\||\lambda_m|^{-1} = \|\mathbf{x}_m\|e^{-y_m}$, we deduce that $y_m$ needs to converge to zero. In order to deduce that $\theta_i$ need to converge to zero for all $i = 1, \ldots, l$, we need to expand Equation \eqref{v_m_real}. Using that $x_{2i+1,m}=0$ for all $i$, we deduce from Proposition \ref{p.algdata1} that:
$$\mathbf{v}_m= 
   \begin{bmatrix} 
   - \sin \theta_{1, m} x_{2,m}\\
   - \cos \theta_{1, m} x_{2,m}+ \exp(-y_m)x_{2,m}\\
   \vdots\\
   - \sin \theta_{l, m} x_{2l,m}\\
   - \cos \theta_{l, m} x_{2l,m}+ \exp(-y_m)x_{2l,m}
   \end{bmatrix}  \text{ if } k = 2l\in 2\mathbf{N}$$ 
or 
   $$\mathbf{v}_m 
   = \begin{bmatrix} 
    - \sin \theta_{1, m} x_{2,m}\\
   - \cos \theta_{1, m} x_{2,m}+ \exp(-y_m)x_{2,m}\\
   \vdots\\
   - \sin \theta_{l, m} x_{2l,m}\\
  - \cos \theta_{l, m} x_{2l,m}+ \exp(-y_m)x_{2l,m}\\
   0
   \end{bmatrix} \text{ if } k = 2l+1\in 2\mathbf{N}+1,$$
    Observe that the $(2i-1)$--th coordinate of $\mathbf{v}_m$ is given by $- \sin \theta_{i, m} x_{2i,m}$, and the $(2i)$--th coordinate of $\mathbf{v}_m$ is given by $- \cos \theta_{i, m} x_{2i,m}+ \exp(-y_m)x_{2i,m}$. This implies that $\theta_i$ needs to go to zero as fast as the inverse of $x_{2i}$, if $x_{2i}$ diverges. In particular the quantity $x_{2j,m} o(\theta_{j,m})$  limits to zero. Considering a first order Taylor expansion of the involved functions, we can then conclude the first part of the Proposition.

In order to verify the second statement we consider, again, a first order Taylor expansion of the involved functions. Using the little-o notation, we can see that
$$v_{2j-1, m} = \left(-\theta_{j, m} +o(\theta_{j, m}^2)\right) x_{2j,m}$$ 
and 
$$v_{2j, m} =  \left(-y_m  + o(\theta_{j, m}) + o(y_m)\right) x_{2j,m}.$$ 
From this, the result follows.
\end{proof}

We can now construct sequences of cyclic subgroups whose geometric limit contains a subgroup isomorphic to $\mathbf{Z}^{l+1}$. We will come back to these sequences in Section \ref{s.free}, and use them to construct interesting geometric limits of free groups. As a consequence we will deduce that their geometric limit is in fact discrete.
\begin{prop}\label{hyp_Jor}
  Let $\left(\rho_n\co\bZ\to \SO_0(1, k+1)\right)_{n\in\bN}$ be the sequence of convex-cocompact representations such that $\rho_n(1)\in {\sf G}_\infty$ are the well positioned loxodromic elements defined by the geometric data
$$\left(\mathbf{x}_n, \frac{1}{n^{l+1}}, \frac{2\pi}{n}, \ldots, \frac{2\pi}{n^{l}}\right),$$
where 
 $$\bR^k \ni \mathbf{x}_n = (x_1, \ldots, x_k) = \begin{cases}
(0, n \ldots, 0, n)& k = 2l\\
(0,n, \ldots, 0, n, 0)& k = 2l+1
\end{cases}.$$
Then the sequence has a discrete and faithful algebraic limit and its geometric limit contains a group isomorphic to $\bZ^{l+1}$, where $l = \floor*{\frac{k}{2}}$.
\end{prop}

\begin{proof}
It is easy to determine, from the geometric data of $\rho_n(1) =\rho_n(n^0) =\colon g_{0,n} $, the geometric data of $\rho_n(n^r) = g_{r,n}$ for all  $r = 0, \ldots, l$. In fact, one can see that the elements $g_{r,n}$ are defined by the following geometric data:
\begin{itemize}
  \item the second fixed point remains $\mathbf{x}_n = \begin{cases}
(0, n \ldots, 0, n)& k = 2l\\
(0, n \ldots, 0, n, 0)& k = 2l+1
\end{cases};$
  \item the translation length is given by $y_{r,n} = n^{-(l+1-r)} \in \bR_+$;
  \item the rotation angles are given by 
  $$ \theta_{j, r,n} = \frac{2\pi}{n^{j-r}} \text{for } j= 1, \ldots, l.$$
\end{itemize}
 For the convenience of the reader we tabulated the geometric data for the elements $g_{r,n}$ in Table \ref{t.1}. 
\begin{table}[ht]
   \renewcommand*{\arraystretch}{2.3}
    \centering
\begin{tabular}{|l||l|l|l|l|}
\hline
&$\mathbf{x}\in \bR^k$&$\theta_{1}$&$\theta_{j}\;(j = 2, \ldots, l)$&$y$\\
\hline\hline
$g_{0,n}=\rho_n(1)$
&\(\begin{cases}
(0, n \ldots, 0, n)^T& k = 2l\\
(0, n \ldots, 0, n, 0)^T& k = 2l+1
\end{cases}\)
&$\displaystyle\frac{2\pi}{n}$
&\(\displaystyle\frac{2\pi}{n^j}\)&\(\displaystyle\frac{2\pi}{n^l}\)
\\
\hline
\(\begin{aligned}g_{a,n}=\rho_n(n^a) \\{}_{(a = 1, \ldots, l-1)}\end{aligned}\)
&\(\begin{cases}
(0, n \ldots, 0, n)^T& k = 2l\\
(0, n \ldots, 0, n, 0)^T& k = 2l+1
\end{cases}\)
&$0$
&\(\begin{cases}
0&j\leq a\\
\displaystyle\frac{2\pi}{n^{j-a}}&j>a
\end{cases}\)
&\(\displaystyle\frac{1}{n^{l-a+1}}\)\\
\hline
$g_{l,n}=\rho_n(n^l)$
&\(\begin{cases}
(0, n \ldots, 0, n)^T& k = 2l\\
(0, n \ldots, 0, n, 0)^T& k = 2l+1
\end{cases}\)&$0$&$0$&\(\displaystyle\frac{1}{n}\)\\
\hline
\end{tabular}
\medskip
\caption{The geometric data for $\rho_n(n^r) = g_{r, n}\in \SO_0(1, k+1)$ for $r = 0, \ldots, l$.}\label{t.1} 
\end{table}

Using Proposition \ref{p.converge} we deduce that the sequences $(g_{r,n})_{n\in\N}$ converge to the elements 
$$g_r:= \begin{bmatrix} 
   1  & \mathbf{v}_r^T  & \frac{1}{2}|\mathbf{v}_r|^2\\
   0 & I & \mathbf{v}_r\\
   0 & 0 & 1
   \end{bmatrix}$$
where the vectors $\mathbf{v}_r$ are the limits, as $n$ goes to infinity, of the vectors $\mathbf{v}_{r,n}$ appearing in the algebraic data of the element $g_{r,n}$.
As a result we deduce that the geometric limit of the sequence $(\rho_n(\bZ))_{n\in\bN}$ contains the group, isomorphic to $\Z^l$ generated by the matrices $g_0, \ldots, g_l$.
 For the convenience of the reader we tabulated such values in Table \ref{t.2}. Note that the last column exists only when $k = 2l+1$ is odd.  When $k$ is even, one should read the table, ignoring the last column.

\begin{table}[ht]
\renewcommand*{\arraystretch}{1.3}
\begin{center}
\begin{tabular}{|l||l|l|l|}
\hline
&$(v_{1}, v_{2})$&$(v_{2j-1}, v_{2j}) \; (j = 2, \ldots, l)$& $v_{2l+1}$ \\
\hline\hline$g_0$&$(-2\pi, 0)$&$(0, 0)$& $0$\\
\hline
$g_{a} \;{}_{(a = 1, \ldots, l-1)}$
&$(0,0)$&\(\displaystyle\left\{\begin{array}{ll}(0, 0)&j\neq a+1\\(-2\pi, 0)&j=a+1\end{array}\right.\)& $0$\\
\hline
$g_{l}$&$(0, -1)$&$(0, -1)$& $0$\\
\hline
\end{tabular}
\medskip
\caption{The vectors $\mathbf{v}_r = (v_{1,r}, \ldots, v_{k,r})$ for $r = 0, \ldots, l$  defining the matrices $g_0, \ldots, g_l\in \SO_0(1,k+1)$.}\label{t.2}
\end{center}
\end{table}

\end{proof}
Proposition \ref{hyp_Jor} can be easily modified to obtain the following:
\begin{Corollary}
  Let $l = \floor*{\frac{k}{2}}$. For any $i \in \{ 1, \ldots, l+1 \}$, there exists a sequence $\left(\rho_n\co\bZ\to \SO_0(1,k+1)\right)_{n\in\bN}$ of convex-cocompact representations whose algebraic limit is a discrete and faithful representation $\rho_{\infty}\co\bZ\to \SO_0(1,k+1)$ and such that the geometric limit $\Gamma_G$ contains a group isomorphic to $\bZ^{i}$.
\end{Corollary}

\subsection{Complex hyperbolic space}\label{complex_jor}

Let  ${\sf G} = \SU(1, k+1)$ be the group of isometries of the complex hyperbolic space $\HH^{k+1}_{\mathbf{C}}$, which we can identify with the Riemannian symmetric space $\HH^{k+1}_{\mathbf{C}} = \faktor{\SU(1, k+1)}{\mathsf{U}(k+1)}$. 

As above, we first define the geometric data associated to a suitable loxodromic element $g\in\SU(1,k+1)$. We use the model for $\HH^{k+1}_{\mathbf{C}}$ discussed in Section \ref{rank_one}, where we consider the Hermitian form $\Psi$ of signature $(1, k+1)$ defined by the matrix 
$Q = \begin{bmatrix}
 &  & 1\\
 & -E_{k} &\\
1 & &  
\end{bmatrix}$. We define
$$\SU(1, k+1) = \{A \in \GL(k+2, \bC) \mid A^* Q A = Q, \mathrm{det}(A) = 1\}.$$ 
The compact centralizer of the  geodesic with endpoints $0$ and $\infty$ is then given by 
$$ \mathsf{U}(k) \cong \left\{B = \begin{bmatrix}
e^{i\phi} &  & \\
 & A &\\
 & &  e^{i\phi}
\end{bmatrix} \in \GL(k+2, \bC)\left|\; \begin{array}{l}A^* A =  E_{k},\\ \mathrm{det}(A) = e^{-2i\phi}\end{array}\right. \right\} < \SU(1, k+1).$$ 
Using the notation defined in Section \ref{rank_one}, the stabilizer of $\infty$ will be the model for our parabolic subgroup of $G$ and we already know that 
$$ {\sf G}_\infty = \mathrm{Stab}_G(\infty) =\left\{ 
 B = \begin{bmatrix} 
 \lambda e^{i\phi}  & \lambda e^{i\phi} \bar{a}^T A & \lambda e^{i\phi} (\frac{|a|^2}{2}+ it)\\
 0 & A & a\\
 0 & 0 & \lambda^{-1}e^{i\phi}
 \end{bmatrix} \left|
\begin{array}{l} 
 \lambda \in \bR_{+}, \\ a\in \bC^k, \\A\in \mathsf{U}(k),\\ \mathrm{det}(A) = e^{-2i\phi}
 \end{array}\right.  \right\}.$$ 

Every loxodromic element $B$ in ${\sf G}_\infty$ can be uniquely determined by:
\begin{itemize}
  \item its second fixed point $ (\mathbf{x}, t) \in \Heis_{k} =\bC^{k}\rtimes \bR \cong \partial\HH^{k+1}_{\mathbf{C}}\setminus\{\infty\}$;
  \item its translation length $y\in\bR$, for which we set $\lambda = e^y\in(0, \infty)$;
  \item its elliptic part, which corresponds to a matrix $A\in \mathsf{U}(k)$.
\end{itemize}  

 In addition, the matrix $A \in \mathsf{U}(k)$ can be diagonalised, which means that, up to choosing a suitable basis of $\bC^k$, we can assume that $A$ is of the form:
\begin{equation}\label{e.AC}
A = \begin{bmatrix} 
     e^{i\theta_1} & &\\
     & \ddots &  \\
     &  & e^{i\theta_k}\\
\end{bmatrix},
\end{equation}
where $\theta_i \in \bR/2\pi\bZ$.
To summarize we have the following definitions.  

\begin{Definition}[Well positioned $g\in\SU(1,k+1)$]
Let $g\in\SU(1,k+1)$ be a loxodromic element with Iwasawa decomposition $g=h_ge_g$.  We say that $g$ is \emph{well positioned} if the attractive fixed point of $g$ is $\infty$ and that the elliptic element is  diagonal in the standard basis. 
\end{Definition}
\begin{Definition}[Geometric Data for well positioned $g\in\SU(1,k+1)$]\label{d.gdata2}
The \emph{geometric data} associated to a well positioned loxodromic element $g\in\SU(1,k+1)$ is the $(2k+2)$--tuple
$$\left(\mathbf{x}, t, y, \theta_1, \ldots, \theta_k\right) \in (\bC^k \rtimes \bR) \times \bR_+ \times  (\bR/2\pi\bZ)^k$$ 
where 
\begin{itemize}
\item $(\mathbf{x},t)\in \bC^k \times \bR=\Heis_k$ are the coordinates of the repulsive fixed point of $g$ (cfr. Section \ref{s.bdry}), 
\item $y$ is the translation length of $h_g$, 
\item $\theta_i$ are the rotation angles of $e_g$.
\end{itemize}
In this case we set 
$$\phi = -\frac{\theta_1+ \cdots +\theta_k}{2}.$$
\end{Definition}
\begin{Remark}
Any loxodromic element $g\in\sf{SU}(1,k+1)$ is conjugated to a well positioned element $g^k$ through an element $k\in\sf K$. Furthermore there exists a unique such conjugate $g^k$ if we additionally require that $x_{1}\leq x_{2}\ldots\leq x_{k}$ and $0\leq\theta_i\leq  \theta_{i+1}<2\pi$ if $x_{i}=x_{i+1}$.
\end{Remark}
It is easy to compute the \emph{algebraic data} corresponding to a given geometric data, namely the matrix representing the element $g$:
\begin{prop}\label{p.algdata2}
If $g\in\SU(1,k+1)$ is a well positioned loxodromic element with geometric data $\left(\mathbf{x}, t, y, \theta_1, \ldots, \theta_k\right) \in (\bC^k \rtimes \bR) \times \bR_+ \times  (\bR/2\pi\bZ)^k$, then the matrix $M_{(\mathbf{x}, t, y, \theta_1, \ldots, \theta_k)}$ representing $g$ is given by 
$$M_{(\mathbf{x}, t, y, \theta_1, \ldots, \theta_k)} = \begin{bmatrix} 
        \lambda e^{i\phi} & \lambda e^{i\phi} \bar{\mathbf{v}}^T A & s \\
        0 & A & \mathbf{v}\\
        0 & 0 & \lambda^{-1}e^{i\phi}
        \end{bmatrix},
$$
where 
\begin{itemize}
\item $A=A_{\theta_1, \ldots, \theta_k}$ is the matrix in Equation (\ref{e.AC}),
\item $\lambda=e^{y}$,
\item $\mathbf{v} = -A \mathbf{x} + \lambda^{-1} \mathbf{x} e^{i\phi} \in \bC^k$,
\item $s =  \frac{1}{2}|\mathbf{x}|^2 e^{i\phi} (\lambda + \lambda^{-1}) - i t e^{i\phi} (\lambda - \lambda^{-1}) - \bar{\mathbf{x}}^T A \mathbf{x}\in\bC$.
\end{itemize}
\end{prop}

\begin{proof}
The element $g$ is conjugated to an element in the stabilizer $L$ of the geodesic between $0$ and $\infty$ through the unique unipotent element $u_{({\bf x},t)}$ fixing $\infty$ and such that $u_{({\bf x},t)}\cdot 0= ({\bf x},t)$. In particular, we have:
$$u_{({\bf x},t)} = \begin{bmatrix} 
    1  & \bar{\mathbf{x}}^T  & \frac{1}{2}|\mathbf{x}|^2+it\\
    0 & 1 & \mathbf{x}\\
    0 & 0 & 1
    \end{bmatrix}.$$
It is easy to compute that $L$ has the following form:
$$ L={\sf G}_\infty \cap \mathrm{Stab}_G(\mathbf{0}) = \left\{ 
  \begin{bmatrix} 
 \lambda e^{i\phi} & 0 & 0\\
 0 & A & 0\\
 0 & 0 & \lambda^{-1} e^{i\phi}
 \end{bmatrix} \left|\begin{array}{l}  \lambda \in \bR_{+}, \\A\in \mathsf{U}(k), \\\mathrm{det}(A) = e^{-2i\phi}\end{array}\right.\right\}.$$
 We can now compute the matrix $M = M_{(\mathbf{x},t, y, \theta_1, \ldots, \theta_k)}$ associated with the geometric data 
 $(\mathbf{x}, t, y,  \theta_1, \ldots, \theta_k):$
 \begin{align*}
  M = M_{(\mathbf{x}, t,y, \theta_1, \ldots, \theta_k)} &=  \begin{bmatrix} 
    1  & \bar{\mathbf{x}}^T  & \frac{1}{2}|\mathbf{x}|^2+it\\
    0 & 1 & \mathbf{x}\\
    0 & 0 & 1
    \end{bmatrix} \begin{bmatrix} 
     \lambda e^{i\phi}  & 0  & 0\\
     0 & A & 0\\
     0 & 0 & \lambda^{-1} e^{i\phi}
     \end{bmatrix} \begin{bmatrix} 
       1  & -\bar{\mathbf{x}}^T  & \frac{1}{2}|\mathbf{x}|^2-it\\
       0 & 1 & -\mathbf{x}\\
       0 & 0 & 1
       \end{bmatrix} \\
       &= \begin{bmatrix} 
        \lambda e^{i\phi} & \lambda e^{i\phi} \bar{\mathbf{v}}^T A & s \\
        0 & A & \mathbf{v}\\
        0 & 0 & \lambda^{-1}e^{i\phi}
        \end{bmatrix},
 \end{align*}
 where 
 \begin{align*}
  \mathbf{v} &= -A \mathbf{x} + \lambda^{-1} \mathbf{x} e^{i\phi} \in \bC^k,\\
  s &=  \frac{1}{2}|\mathbf{x}|^2 e^{i\phi} (\lambda + \lambda^{-1}) - i t e^{i\phi} (\lambda - \lambda^{-1}) - \bar{\mathbf{x}}^T A \mathbf{x}\\
              &= \frac{1}{2}|\mathbf{x}|^2 e^{i\phi} (\lambda - \lambda^{-1}) - i t e^{i\phi} (\lambda - \lambda^{-1}) + \bar{\mathbf{x}}^T \mathbf{v}\in \bC.
 \end{align*}\end{proof}
With this at hand, it is easy to give necessary and sufficient conditions on the geometric data of a sequence of well positioned loxodromic elements $g_n$ guaranteeing that the sequence converges in $\SU(1,k+1)$.
\begin{prop}\label{p.convergeC}
Let $(g_m)_{m\in\bN}\subset \SU(1,k+1)$ be a sequence of  well positioned loxodromic elements with geometric data $$\left(\mathbf{x}_m, t_m, y_m, \theta_{1,m}, \ldots, \theta_{k,m}\right)$$ and let $\mathbf{x}_m = (x_{1, m}, \ldots, x_{k, m})$.
Assume, furthermore, that $x_{k,m}\geq\ldots\geq x_{1,m}$ and $x_{1,m}$ diverges.
Then the sequence converges algebraically if and only if
\begin{itemize}
\item $(\phi_m-\theta_{j,m})x_{j,m}$ converges for all $j = 1, \ldots, k$;
\item $y_m x_{k,m}$ converges;
\item $\sum_{j=1}^k x_{j,m}^2\left(\phi_m -\theta_{j,m}\right)$ converges;
\item $t_m  y_m$ converges.
\end{itemize}
In this case the algebraic limit is the unipotent matrix associated to the pair $({\bf v}, s)$ where ${\bf v} = (v_1, \ldots, v_k)\in \bC^k$ and $s\in \bC$ are defined by
$$\left\{\begin{array}{l}
v_{j}=\displaystyle{\lim_{m \to \infty} [-y_mx_{j,m}+i(\phi_{m}-\theta_{j,m})x_{j,m}]\quad \forall j = 1, \ldots, k}\\
s=\displaystyle{\lim_{m \to \infty}  \frac{1}{2}\sum_{j=1}^k \left[x_{j,m}^2 \left(y_m^2 - \phi_m^2 +\theta_{j,m}^2\right)\right] 
+ i\left( \sum_{j=1}^k \left[x_{j,m}^2\left(\phi_m -\theta_{j,m}\right)\right]+ 2t_m  y_m\right).}
\end{array}\right.$$
\end{prop}

\begin{proof}
In order to understand under which conditions the elements $g_m$ converge as $m$ diverges, we need to compute approximations of the pairs $({\mathbf{v}}_m,s_m)$ determining the matrix $M_m$ associated to the geometric data of $g_m$. 
It follows from Proposition \ref{p.algdata2} that:
\begin{equation}\label{e.conv}
\mathbf{v}_m = -A_m \mathbf{x}_m + \lambda_m^{-1} \mathbf{x}_m e^{i\phi_m} \in \bC^k.
\end{equation} 
Since $\mathbf{v}_m$ is the difference of two vectors of norm $\|A_m\mathbf x_m\| = \|\mathbf x_m\|$ and $| \lambda_m^{-1}| \|\mathbf x_m\| = e^{-y_m} \|\mathbf x_m\|$, and $\|\mathbf x_m\|$ diverges, we deduce that $y_m$ needs to converge to $0$. Furthermore, since Equation \eqref{e.conv}, expressed in coordinates, yields that 
$$v_{j, m} = (e^{-y_m}e^{i\phi_m}-e^{i\theta_{j,m}})x_{j,m} = e^{i\phi_m}x_{j,m}(e^{-y_m}-e^{i(\theta_{j,m}-\phi_m)})\in\bC.$$ 
From that we deduce that also $\phi_m -\theta_{j, m}$ needs to go to zero.
Since $(\phi_m -\theta_{j, m})$ converges to zero as $m\to \infty$ and since $\phi_m = -\frac{\theta_{1, m}+ \cdots +\theta_{k, m}}{2}$, then also $\sum_{j =1}^k\left(\phi_m -\theta_{j, m}\right) = (k+2)\phi_m$ converges to zero as $m\to \infty$, so both $\phi_m$ and $\theta_{j, m}$ converge to zero as $m$ diverges. 

As a result we can use a Taylor expansion to understand the rate of convergence. From Equation \eqref{e.conv} and using the Taylor expansion for $e^{i\phi_m}$, $e^{i\theta_{j, m}}$ and $e^{y_m} \pm e^{-y_m}$ we have that:
   \begin{align*}
    v_{j, m} &= x_{j,m} (\lambda_m^{-1}e^{i\phi_m}-e^{i\theta_{j,m}}) = x_{j,m} (e^{-y_m+i\phi_m}-e^{i\theta_{j,m}})\\
&= x_{j,m} \left(-y_m  +i(\phi_m-\theta_{j,m}) \right) +o\left(1\right),
 \end{align*}  
for $j = 1, \ldots, k$. This shows that:
\begin{itemize}
	\item $y_m$ goes to zero at least as fast as $x_{k,m}^{-1}$, and
	\item $\phi_m -\theta_{j, m}$ needs to go to zero at least as fast as $x_{j,m}$, for every $j = 1, \ldots, l$.
\end{itemize} 
This justifies the first two conditions. In addition, the number $s_m \in\bC$ can be approximated as:
\begin{align*}
s_m &=  \frac{1}{2}|\mathbf{x}_m|^2 e^{i\phi_m} (\lambda_m + \lambda^{-1}_m) - \bar{\bf x}_m^T A {\bf x}_m - i t_m e^{i\phi_m} (\lambda_m - \lambda^{-1}_m)\\
&=  \frac{1}{2}|\mathbf{x}_m|^2 e^{i\phi_m} (2+y_m^2)- \bar{\bf x}_m^T A {\bf x}_m - i t_m e^{i\phi_m} (2y_m) +o(1)\\
&=  \sum_{j=1}^k x_{j,m}^2\left(e^{i\phi_m} -e^{i\theta_{j,m}}\right)+\frac{1}{2}|\mathbf{x}_m|^2 y_m^2- i 2t_m  y_m  +o(1)\\
&=  \left( \frac{1}{2}\sum_{j=1}^k \left[x_{j,m}^2 \left(y_m^2 - \phi_m^2 +\theta_{j,m}^2\right)\right]\right)+ i\left( \sum_{j=1}^k \left[x_{j,m}^2\left(\phi_m -\theta_{j,m}\right)\right]+ 2t_m  y_m\right) +o(1).\\
\end{align*}
This computation allows us to conclude the two last claims, and hence complete the proof.
\end{proof}

The highest rank of a discrete abelian subgroup of $\sf G=\SU(1,k+1)$ is $k+1$. This is what we get in Theorem \ref{t.INTRO1}, using the following proposition and an argument from the very end of the paper.
\begin{prop}\label{hyp_Jor2}
  The  sequence $\left(\rho_n\co\bZ\to \SU(1, k+1)\right)_{n\in\bN} $ of convex-cocompact representations defined by $\rho_n(1) = g_{0, n}\in {\sf G}_\infty$ loxodromic element with geometric data given by 
$$
(\mathbf{x}, t, y,  \theta_1, \ldots, \theta_k)=\left((n,\ldots, n), 0, \frac{1}{n^{k+2}},\theta_1, \ldots, \theta_k\right)$$
where 
$$\left\{\begin{array}{lr
}
\theta_1= \displaystyle{\frac{2\pi}{n}+\frac{2\pi}{n^{k+1}}},\\
\theta_j = \displaystyle{\frac{2\pi}{n^j}}, & 2\leq l\leq k-1\\ 
\theta_k=-\displaystyle{ \sum_{j =1}^{k-1}\frac{2\pi}{n^j}}.
\end{array}\right.
$$
has (discrete and faithful) algebraic limit and its geometric limit $\Gamma_G$ contains a group isomorphic to $\bZ^{k+1}$.
\end{prop}

\begin{proof}
As in Proposition \ref{hyp_Jor}, we first note that it is easy to determine the geometric data of $\rho_n(n^r) = g_{r,n} $ for all $r\geq 0$ from the geometric data of $\rho_n(1)= g_{0,n}$. In fact, the geometric data of $g_{r,n}$ is the following:
\begin{itemize}
\item the second fixed point of $g_{r,n}$ is still given by $(\mathbf{x}_n, t_n)\in \Heis_{k}=\bC^{k}\rtimes \bR$ such that $\mathbf{x}_n = (x_{1,n}, \ldots, x_{k,n})^T = (n, \ldots, n)^T,$ and $t_n=0$;
\item the translation length of $g_{r,n}$ is given by $y_{r,n} = n^{-(k+2-r)}$;
\item the rotation angles $\theta_{j, r,n}$ of $g_{r,n}$ are defined by $\theta_{j, r,n} = n^r\theta_{j}\in\bR/2\pi\bZ$. 
\end{itemize}
Also, for the convenience of the reader, we tabulated the geometric data for $g_{r, n}$ in Table \ref{t.3}.
\begin{table}[ht]
   \renewcommand*{\arraystretch}{2.3}
    \centering
\begin{tabular}{|l||l|l|l|l|l|l|l|}
\hline
&$\mathbf{x} $&$t$&$\theta_{1}$&$\theta_{j}\;(j = 2, \ldots, k-1)$&$\theta_{k}$&$\phi$&$y $\\
\hline\hline
$g_{0,n}$&$(n, {\scriptstyle\ldots}, n)$&$0$&\(\displaystyle\frac{2\pi}{n}+\frac{2\pi}{n^{k+2}}\)&\(\displaystyle\frac{2\pi}{n^j}\)&\(\displaystyle-\sum_{q=1}^{k-1} \frac{2\pi}{n^q}\)&$\displaystyle\frac{-\pi}{n^{k+2}}$&\(\displaystyle\frac{1}{n^{k+3}}\)\\
\hline
$g_{a, n} \;{}_{(a = 1, \ldots, k-1)}$ &$(n, {\scriptstyle\ldots}, n)$&$0$&\(\displaystyle\frac{2\pi}{n^{k+2-a}}\)&\(\displaystyle\left\{
{  \renewcommand*{\arraystretch}{1}
\begin{array}{ll}0&{\scriptstyle j\leq a}\\\displaystyle\frac{2\pi}{n^{j-a}}&{\scriptstyle j>a}\end{array}}\right.\)&\(\displaystyle-\sum_{q=a+1}^{k-1} \frac{2\pi}{n^{q-a}}\)&$\displaystyle\frac{-\pi}{n^{k+2-a}}$&\(\displaystyle\frac{1}{n^{k+3-a}}\)\\
\hline
$g_{k-1,n}$&$(n, {\scriptstyle\ldots}, n)$&$0$&\(\displaystyle\frac{2\pi}{n^{2}}\)&$0$&$0$&\(\displaystyle\frac{-\pi}{n^{2}}\)&\(\displaystyle\frac{1}{n^{3}}\)\\
\hline
$g_{k+1,n}$&$(n, {\scriptstyle\ldots}, n)$&$0$&$0$&$0$&$0$&$0$&\(\displaystyle\frac{1}{n}\)\\
\hline
\end{tabular}
\medskip
\caption{The geometric data for $\rho_n(n^r) = g_{r,n}$ for
 $r = 0, \ldots, k-1, k+1.$}\label{t.3}
\end{table}

Using Proposition \ref{p.convergeC} we deduce that the geometric limit of the sequence $(\rho_n(\bZ))_{n\in\bN}$ contains the group generated by the matrices $g_0, \ldots, g_{k-1}, g_{k+1}$, where
$$g_r:= \begin{bmatrix} 
   1  & \overline{\mathbf{v}_r}^T  & s_r\\
   0 & I & \mathbf{v}_r\\
   0 & 0 & 1
   \end{bmatrix} \;\; \text{ for } r = 0, \ldots, k-1, k+1,$$
and the vectors $\mathbf{v}_r$ are the limit, as $n$ goes to infinity, of the vector $\mathbf{v}_{r,n}$ appearing in the algebraic data of the element $g_{r,n}$ and $s_r\in \bC$ is the limit, as $n$ goes to infinity, of the number $s_{r, n}$ again appearing in the algebraic data of $g_{r,n}$. For the convenience of the reader we tabulated such values in Table \ref{t.4}. 
\begin{center}
\begin{table}[ht]
   \renewcommand*{\arraystretch}{1.7}
    \centering
\begin{tabular}{|l||l|l|l|l|}
\hline
&$v_{1}$&$v_{j} \;(j = 2, \ldots, k-1)$&$v_{k}$&$s$\\
\hline\hline
$g_0$&$-2\pi i$&$0$&$2\pi i$&$4\pi^2$\\
\hline
$g_{a, n} \;{}_{(a = 1, \ldots, k-1)}$ 
&$0$&\(\displaystyle\left\{
{  \renewcommand*{\arraystretch}{1}
\begin{array}{ll}0&j\neq a+1\\\displaystyle{-2\pi i}&j=a+1\end{array}}\right.\)
&$2\pi i$&$4\pi^2$\\
\hline
$g_{k-1}$&$0$&$0$&$0$&$-(k+2)i\pi$\\
\hline
$g_{k+1}$&$-1$&$-1$&$-1$&$\frac{k}{2}$\\
\hline
\end{tabular}
\medskip
\caption{The vectors $\mathbf{v}_r = (v_{1, r}, \ldots, v_{k, r})\in \bC^k$ and $s_r\in \bC$ defining the matrices $g_r \in \SU(1,k+1)$ for $r = 0, \ldots, k-1, k+1$.}\label{t.4}
\end{table}
\end{center}
\end{proof}

Proposition  \ref{hyp_Jor2}  can be easily modified to obtain the following:
\begin{Corollary}
For any $i = 1, \ldots, k+1$, there exists a sequence of Anosov representations $\{\rho_n\co\bZ\to \mathrm{Isom}^+(\HH^{k+1})\}_{n\in\N}$  whose algebraic limit is a discrete and faithful representation $\rho_{\infty}$ and such that the geometric limit $\Gamma_G$ of the subgroups $\rho_n(\bZ)$ contains a group isomorphic to $\bZ^{i}$.
\end{Corollary}

\subsection{Examples of sequences with non-discrete geometric limit}\label{s.counter}
We conclude the section providing two examples of sequences whose geometric limit is not discrete. In both our examples we have ${\sf G}=\SO_0(1,k+1)$, but it's easy to construct similar examples for any other Lie group $\sf G$.
\begin{Example}(Non-faithful algebraic limit)
Let $\theta\in \bR/2\pi \bZ$ be an irrational multiple of $\pi$, and let $n_r\in\bZ$ be such that $n_r\theta\to 0$ (of course $n_r$ is unbounded). We consider the well positioned loxodromic elements $\rho_n(1)$ with geometric data 
$$\left(\mathbf{x_r}, y_r, \theta_{1,r}, \ldots, \theta_{l,r}\right)=\left(0, 1/{n_r^2}, n_r\theta,0 \ldots, 0\right).$$
The limit of the sequence of convex-cocompact representations  $\left(\rho_n\co\bZ\to \SO_0(1, k+1)\right)_{n\in\bN}$ is the trivial representation, which is in particular not faithful. The geometric limit is the whole one parameter group of rotations containing the rotation associated to $A_{\theta,0 \ldots, 0}$ which is, in particular, not discrete.
\end{Example}
With the same idea one can construct examples of sequences of representations whose geometric limit is a torus of dimension $l$ if ${\sf G}=\SO_0(1,k+1)$ and of dimension $k$ if ${\sf G}=\SU(1,k+1)$ .

The following example shows that requiring that the algebraic limit is faithful with discrete image doesn't guarantee that the geometric limit is discrete. Note that this example is not uniformly discrete and so it does not contradicts Corollary \ref{disc_Z}.

\begin{Example}(Discrete and faithful algebraic limit, non-discrete geometric limit)
In this case we consider a minor modification of the discussion in the proof of Proposition \ref{p.algdata1} and in example in Proposition \ref{hyp_Jor}. Namely  we consider  the  sequence of convex-cocompact representations $\left(\rho_n\co\bZ\to \SO_0(1, k+1)\right)_{n\in\bN} $ such that the geometric data of $\rho_n(1)$ is given by 
$$\left(\mathbf{x}, \frac{1}{n^{2l+1}}, \frac{2\pi}{n^2}, \ldots, \frac{2\pi}{n^{2l}}\right)$$
where 
 $$ \mathbf{x} = (x_1, \ldots, x_k)^T = \begin{cases}
(0,n, \ldots, 0,n)^T& k = 2l\\
(0,n, \ldots, 0,n, 0)^T& k = 2l+1
\end{cases} \in \bR^k.$$
Such sequence of representations has  discrete and faithful algebraic limit, but the geometric limit $\Gamma_G$ of the subgroups $\rho_n(\bZ)$ is isomorphic to $\bZ\times \bR^{l}$, where $l = \floor*{\frac{k}{2}}$. Indeed the geometric limit is, by definition, closed, and in this example contains the group elements associated to every rational multiple of the vector $-2\pi e_{2s-1}$  for $1\leq s\leq l$. Indeed if $p$ and $q$ are coprime integers, and we set $m_n$ to be the closest integer to $pn^2/q$, then  the unipotent element associated to $-2\pi p/q e_{2s-1}$ is the limit of $\rho_n(m_nn^{2(s-1)})$.

\end{Example}

\section{Classification of geometric limits for cyclic representations in $\SO_0(1,k+1)$}\label{thats_all}
The goal of the section is to classify all possible geometric limits of cyclic sequences generated by a loxodromic element in $\SO_0(1,k+1)$. We begin with a general criterion guaranteeing that the convergence is strong.

\begin{prop}\label{p.strongconvergence}
For all $n \in \bN$, let $\rho_n\co\bZ\to \sf G$ be a convex-cocompact representation from $\bZ$ into a rank 1 group ${\sf G}$. Suppose $(\rho_n)_{n \in \bN}$ converges algebraically to a discrete and injective representation $\rho_\infty$ and geometrically to $\Gamma_G$. If, up to passing to a subsequence, there exists $N \in \bN$ such that $(\rho_n(1) =\colon g_n)_{n>N}$ is loxodromic and the fixed points $(g_n^+)_{n>N}$ and $(g_n^+)_{n>N}$ limit to two distinct points, then $(\rho_n)_{n \in \bN}$ converges strongly.
\end{prop}

\begin{proof}
Since the sequences of fixed points $(g_n^+)_{n \in \bN}$ and $(g_n^-)_{n \in \bN}$ converge to a pair of distinct points $g_\infty^+$ and $g_\infty^-$, the generator $g_\infty=\rho_\infty(1)$ of the algebraic limit  $\rho_\infty$ of the sequence fixes $g_\infty^+$ and $g_\infty^-$ as well; since by assumption $\rho_\infty$ is discrete and faithful, $\rho_\infty(1)$ necessarily has positive translation length $\ell(g_\infty) \in (0, \infty)$. Recall that $$\ell(g_\infty) := \mathrm{inf}_{x \in \XG} d_{\XG}(x, g_\infty x).$$
As noted in Proposition \ref{subset}, we know that $\rho_\infty(\bZ) \leq \Gamma_G$, thus we only have to show that any element in $\Gamma_G$ is contained in $\rho_\infty(\bZ)$. 

By definition, for any element $\gamma$ in $\Gamma_G$ there is a sequence $(m_n)_{n \in \bN}$ such that $\gamma=\lim_n \rho_n(m_n)$.  Since the translation length is additive, and, in our case of interest, the translation length of the limit is the limit of the translation lengths, we have that, for $\epsilon \in (0, \ell(g_\infty))$, there exists $N_\epsilon\in \bN$ such that for all $n > N_\epsilon$ we have
$$m_n (\ell(g_\infty)-\epsilon) < m_n \ell(g_n) = \ell(\rho_n(m_n)) < \ell(\gamma)+1.$$
Here the first inequality comes from the fact that $g_\infty = \lim_n g_n$ and the last one comes from the fact that $\gamma = \lim_n \rho_n(m_n)$. Since $(m_n)_{n \in \bN}$ is bounded, we know that $\gamma \in \rho_\infty(\bZ)$.
\end{proof}
The same argument applies for any semisimple Lie group $\sf G$ if we can guarantee that the limit is loxodromic:
\begin{prop}
Let $\sf G$ be a semisimple Lie group, and $(\rho_n\co\bZ\to \sf G)_{n\in\N}$ be a sequence of representations converging algebraically to $\rho_\infty$. If $\rho_{\infty}(1)$ is loxodromic, the convergence is strong. 
\end{prop} 

\begin{remark}\label{r.parabolic1}
More generally one can prove that if $\rho_n(1)$ is parabolic, can be written as $\rho_n(1)=p_ne_n$ for unipotent elements $p_n$ and elliptic elements $e_n$ (cfr. Section \ref{s.isom}), and the sequence $p_n$ converges to a non-zero unipotent element $p_\infty$, then the convergence is strong. On the other hand, if $\rho_n(1)$ is a parabolic which is not a unipotent element, one can create example similar to the one discussed in the previous section. See also Remark \ref{parab_rk}.
\end{remark}

Proposition \ref{conjug} allows us to use the geometric data developed in Section \ref{sec_gen_jor} to study general discrete limits and prove the following:

\begin{prop}\label{p.thatsallR}
Let $(\rho_n\co\bZ\to \SO_0(1,k+1))_{n\in\bN}$ be a sequence of convex-cocompact representations converging geometrically to a discrete group $\Gamma_G$. Then the rank of $\Gamma_G$ is at most $l+1$.
\end{prop}

\begin{proof}
Up to conjugating with elements $k_n\in\sf K$, which doesn't change the rank of the geometric limit (cfr. Proposition \ref{conjug}) we can assume that the loxodromic elements $\rho_n(1)$ are well positioned, and in particular always fix $\infty$. By Proposition  \ref{p.strongconvergence} we can furthermore reduce to the case that the attracting and repelling fixed point converge to the same limit $g_\infty^{\pm}$.

As a result, we can encode the element $g_n$ through its  \emph{geometric data} (as in Definition \ref{d.gdata1}), namely through the $(k+l+1)$--tuple
$$\left(\mathbf{x}_n, y_n, \theta_{1,n}, \ldots, \theta_{l,n}\right) \in \bR^k \times \bR_+\times  (\bR/2\pi \bZ)^l.$$
We can furthermore assume, up to finding better conjugating elements $k_n$, that $x_{2i-1,n}=0$ for all $i=1, \ldots, l$, as this amounts to conjugating with a suitable matrix of the form $A$, and that $x_{2,n}\leq\ldots\leq x_{2l,n}$.  

Assume first that $x_{2,n}$ (and therefore all $x_{2i,n}$) go to infinity. In this case one sees from Proposition \ref{p.converge} that every element in $\Gamma_G$ belong to the unipotent radical in the stabilizer of $g^\pm_\infty$, which, as explained in Section \ref{s.bdry}, we identify to $\bR^{k}$ associating a vector ${\bf v}\in{\bR}^k$ to the corresponding element in $N$. It follows directly from Proposition \ref{p.converge} that the only $\bf v$ that can occur belong to the $(l+1)$--dimensional subspace generated by the first $l$ odd basis vectors together with the vector ${\bf x}_\infty=\lim_n {\bf x}_n/\|{\bf x}_n\| $. As we are assuming that the geometric limit is discrete, the result follows.

In the general case, let $i$ be the maximal index such that $x_{2i,n}$ is bounded. Then the limit of $\exp(-y_n)x_j$, as $n$ goes to $\infty,$ is equal to $x_j$ for all $j\leq 2i$ and the geometric limit is contained in the product of a $i$--dimensional torus and a $(l-i+1)$--dimensional vector space. Here the torus is contained in the stabilizer of  the geodesic with endpoints the point $g_\infty^\pm$ and the point $\bf w$ with
 $$\left\{\begin{array}{ll}
w_{2j}=\lim_n x_{2j,n} &\text{ if } s=2j, 1\leq j\leq i, \\
w_{s}=0&\text{ otherwise}
\end{array}\right.$$
while the $(l-i+1)$--dimensional vector space is generated by the vectors $e_{2s}$ for $i+1\leq s \leq l$  and the vector ${\bf x}_\infty=\lim_n {\bf x}_n/\|{\bf x}_n\| $. In particular, the rank of $\Gamma_G$ is, in this case, at most $l-i+1$.
\end{proof}
\begin{Remark}\label{parab_rk}
One can run a similar analysis in the case in which $\rho_n(1)$ is parabolic, after developing an analogue theory of geometric data for parabolic elements. In this case the rank of the limit is at most $\lfloor\frac{k-1}{2}\rfloor$, and again the elliptic parts in the Iwasawa decomposition are the sole responsible for new elements in the limit.
\end{Remark}

We can now prove the main result of the section: any discrete, torsion free subgroup  satisfying the condition expressed in Proposition \ref{p.thatsallR} can be obtained as a geometric limit of a suitable sequence.

\begin{Theorem}\label{p.allarisesR}
Let $N=N_\bR$ be the unipotent radical of the stabilizer of a point $p\in\partial_\infty\HH^{k+1}_\bR$ and let $\Delta<N$ be discrete, torsion free, and of rank at most $l+1$. Then there is a sequence of convex cocompact representations $(\rho_n\co\bZ\to \SO(1,k+1))_{n\in \bN}$ converging algebraically and such that the geometric limit $\Gamma_G$ is $\Delta$. 
\end{Theorem}

\begin{proof}
We will deal only with the case in which the rank of $\Delta$ is $l+1$, since the other cases are entirely analogue.  We can assume, up to conjugation in $\sf K$, that $N$ is the unipotent radical of the stabilizer of the point at infinity, and, as discussed in Section \ref{s.bdry}, we identify it with $\bR^k$. 

The first step of the proof consists of showing that, up to conjugation in $\sf K$ and applying Proposition \ref{l.dense}, we can reduce to the case in which the group $\Delta$ has a particularly simple basis. Choose generators $v_1,\ldots, v_{l+1}$ of $\Delta$. Up to conjugation in $\sf K$, we can furthermore assume that, for any $1\leq s\leq l$, the space $\langle v_1,\ldots, v_s\rangle$ is contained in the span $V_{\rm odd}^s$ of the first $s$ odd basis vectors of the standard basis and  $v_{l+1}$ is the sum of a vector $v_{l+1}^{\rm odd}\in V_{\rm odd}^l$ and a vector $v_{l+1}^{\rm even}$ in the span $V_{\rm even}^l$ of the even basis vectors. 

We will first arrange for  $v_1,\ldots, v_l$ to be in a standard form that we now define. Let $w_{i,j}\in\bR$ be such that 
$$v_{i}=-2\pi\sum_{j=1}^{i} w_{i,j}e_{2j-1}.$$ 
Thanks to the first part of Proposition \ref{l.dense}, it is enough to show that $\Delta$ arises as a geometric limit under the further assumption that $w_{i,j}\in\bQ$. Up to choosing a different basis of the same lattice,  we can assume that $w_{i,j}=b_{i,j}w_{j,j}$ with $|b_{i,j}|<1$ for $l\geq i>j\geq 1$, and restrict to the sequence of $n$ such that $nb_{i,j}\in\bZ$ for all $1\leq j < i \leq l$. In fact if $c$ is the least common multiple of the denominators of $b_{i,j}$, over all $l\geq i>j\geq 1$, this amounts to choosing $n$ in $c\bZ$. This implies that, in the orthonormal basis  $\{e_1,e_3,\ldots, e_{2l-1}\}$ for $V_{\rm odd}^l$, we have 
$$\begin{array}{rl}
v_1&=-2\pi (w_{1,1},0,\ldots,0)\\
v_2&=-2\pi (w_{1,1}b_{2,1},w_{2,2},\ldots,0)\\
\vdots\\
v_l&=-2\pi (w_{1,1}b_{l,1},w_{2,2}b_{l,2},\ldots,w_{l,l}).\\
\end{array}$$

We will now deal with the vector $v_{l+1}$. Observe, first, that, up to conjugating by an element of $\sf K$ fixing $v_1,\ldots, v_l$, we can assume that there exists $c_{l+1}\in\bQ$ such that 
$$v_{l+1}^{\rm even}=c_{l+1}(w_{1,1},\ldots, w_{l,l}).$$ 
Here the coordinates are taken with respect to the orthonormal basis $\{e_2,e_4,\ldots, e_{2l}\}$ of $V_{\rm even}^l$.
Furthermore, we let $b_{l+1,j}$ be such that  
$$v_{l+1}^{\rm odd}=-2\pi(w_{1,1}b_{l+1,1},w_{2,2}b_{l+1,2},\ldots,w_{l,l}b_{l+1,l}).$$

As a second step, we set, for every $2\leq i\leq l$, the number $d_i$ to be the least common multiple of the denominators of the rational numbers $b_{i,j}$ where $l\geq i>j\geq 1$, and define
 $$D_i=\left\{\begin{array}{ll}\prod_{k\leq i} d_k &i\geq 2\\ 1& i=0,1.\end{array}\right.$$
We then consider the sequence $(\rho_n\co\bZ\to\SO_0(1, k+1))_{n\in\bN}$ defined by $\rho_n(1) = g_n$ loxodromic element determined by the geometric data 
$ \left(\mathbf{x}_n, y_n, \theta_{1,n}, \ldots, \theta_{l,n}\right) $
such that
$$
\begin{array}{ll}
x_{2i-1,n}=0 &\forall i = 1, \ldots, l,\\
x_{2l+1,n}=0 &\text{if $k$ is odd},\\
\displaystyle{x_{2i,n}=nw_{i,i}} &\forall i = 1, \ldots, l,\\
y_n=\displaystyle{\frac{c_{l+1}}{D_l n^{l+1}}},\\
\displaystyle{\theta_{i,n}=\frac{2\pi }{D_{i-1}n^{i}}+\sum_{j=i+1}^{l+1} b_{j,i}\frac{2\pi }{D_{j-1}n^{j}}}&\forall i = 1, \ldots, l.
\end{array}
$$
We can directly compute that
$$D_rn^r y_{n} = \displaystyle{\frac{c_{l+1}}{d_{r+1}\cdots d_{l} n^{l-r+1}}},$$
and an approximate expression for $D_rn^r\theta_{j,n}$ is given by 
$$D_rn^r\theta_{j,n} = \begin{cases}
\displaystyle{b_{r+1,j}\frac{2\pi}{n}+o\left(\frac{1}{n^2}\right)}&\text{if $j\leq r$,}\\
\displaystyle{\frac{2\pi}{n}+o\left(\frac{1}{n^2}\right)}&\text{if $j=r+1$,}\\
\displaystyle{o\left(\frac{1}{n^2}\right)}&\text{if $j\geq r+2$.}
\end{cases}
$$

Thus, for $r=0,\ldots, l-1$, the limit of $\rho_n(D_rn^r)$ is the element in $N$ associated to the vector $v_{r+1}$. Furthermore, the geometric data associated to the element $\rho(D_ln^l)$ is given by
$$\left({\bf x}_n, \displaystyle{\frac{c_{l+1}}{n}}, b_{l+1,1}\frac{2\pi}{n},\ldots, b_{l+1,l}\frac{2\pi}{n}\right).$$
It follows from Proposition \ref{p.converge} that the limit of $\rho_n(D_ln^l)$ is associated to the vector $v_{l+1}$.

Lastly, it remains  to check that the geometric limit $\Gamma_G$ of the sequence $(\rho_n(\bZ))_{n\in\bN}$ is not bigger than $\Delta$. Assume by contradiction that the geometric limit $\Gamma_G$ of $(\rho_n(\bZ))_{n\in\bN}$ strictly contains $\Delta$.   We know from Proposition \ref{p.thatsallR} that any element in $\Gamma_G$ belongs to the subgroup of the unipotent radical $N$ associated to the subspace $\bR^{l+1}$ spanned by the vectors $v_1,\ldots,v_{l+1}$. We consider the strict fundamental domain $F$  for the action of $\Delta$ on $\bR^{l+1}$, given by the union of the interior of the parallelepiped determined by the vectors $v_1,\ldots, v_{l+1}$ together with the interior of all the faces containing the origin. Since, by assumption, the group $\Gamma_G$ is strictly bigger then $\Delta$, it will necessarily contain an element associated to a vector $f\in F$.

In the basis $v_1,\ldots,v_{l+1}$ we have, by definition of $F$, that 
$$f=\sum_{i=1}^{l+1} a_i v_i$$
for some $0\leq a_i<1$. Let $B_{\epsilon}$ denote a small ball around the vector $f$ chosen in such a way that, for all $i$ and every vector $w\in B_{\epsilon}$, all coordinates in the basis $v_1,\ldots,v_{l+1}$ have absolute value smaller than 1.
 
By definition of geometric limit we can write the element $g\in\Gamma_G$ associated to the vector $f$ as the limit of $\rho_n(m_n)$ for a sequence of integers $(m_n)_{n\in\bN}$.  Proposition \ref{l.convtoU} implies  that, for every $n$ big enough, $\rho_n(m_n)0\in B_{\epsilon}$. Observe that the component in the direction of $v_{l+1}^{even}$ of every vector in $B_{\epsilon}$ is smaller than 1, on the other hand such component for the point  $\rho_n(m_n)0$ is $\|x\|(e^{m_ny_n}-1)$. We deduce from this that $m_n<D_ln^{l}$. 

In particular we can write
$$m_n=\sum_{i=0}^{l-1}m_{i,n}n^iD_i$$
with $m_{i,n}<nd_{i+1}$.

Observe that, if $\rho_n(m_n)0\in B_\epsilon$, then, for every $j$, we have
$$m_n\theta_ {j,n}n < 1.$$ This equation for  $j=l$ yields that $m_{l-1,n}=0$.

Using the same argument we conclude, by induction on $k$,  that $m_{l-k,n}=0$ for every $k$, which leads to the desired contradiction.
\end{proof}

\section{Examples of geometric limits of free groups}\label{s.free}

The goal of the section is to show that the examples constructed in Section \ref{sec_gen_jor} also arise as restriction to cyclic subgroups of geometric limits of sequences of representations of free groups,  thus proving Theorem \ref{it.3}. We will then use this result to conclude the proof of Theorem \ref{t.INTRO1} by showing that the subgroups defined in Propositions \ref{hyp_Jor} and \ref{hyp_Jor2} correspond to the geometric limit of those representations.

For the proof we will use an idea discussed by Thurston \cite{thurston-notes}. See also Kapovich \cite{Kap}. Let $(\rho_n\co\bZ \to \sf G)_{n \in \bN}$ denote the generalized J{\o}rgensen sequence discussed in Proposition \ref{hyp_Jor} and \ref{hyp_Jor2}. We will show that, for sufficiently big $n$, there exist fundamental regions for the action of $\langle \rho_n(1) = g_n \rangle$ on $\partial \HH^k_{\mathbf F}$ containing a fixed ball $B$. This will then allow us to find an element $h \in \sf G$ such that the complement of $B$ is contained in a fundamental region for $h$. The representations $\varrho_n\co F_2 =\langle a, b\rangle \to \sf G$ such that $\varrho_n(a) = g_n$ and $\varrho_n(b) = h$ are then convex cocompact and provide our desired example. As always we will deal with the real and complex case separately.

\begin{Theorem}\label{t.Schotky_real}
There exists a sequence $\left(\varrho_n\co F_2\to \SO_0(1,k+1)\right)_{n \in \bN}$ of convex-cocompact representations such that the restriction $\varrho_n|_{\langle a\rangle}$ is the sequence of representations constructed in Proposition \ref{hyp_Jor}. 
\end{Theorem}

\begin{proof}
Let $\left(\rho_n\co\bZ\to \SO_0(1,k+1)\right)_{n \in \bN}$ denote the sequence of representations constructed in  Proposition \ref{hyp_Jor}, and let $g_n = \rho_n (1)$. Recall that $g_n$ is loxodromic with fixed points $\infty$ and $\mathbf x = (0,n,0,n \ldots)$, and  its geometric data is given by
$$\left(\mathbf{x}, \frac{1}{n^{l+1}}, \frac{2\pi}{n}, \ldots, \frac{2\pi}{n^{l}}\right).$$

We will show that, for sufficiently big $n$, there exist fundamental regions for the action of $\langle \rho_n(1) = g_n \rangle$ on $\partial \HH^k_{\mathbf R}$ containing a fixed ball $B$. To do that, we introduce  polar coordinates on  $\bR^k = \partial \HH^k \setminus \{\infty\}$ centered at the fixed point $\mathbf x$: 
these are given by
$$\begin{array}{ll}
(r_1,\vartheta_1,\ldots, r_l,\vartheta_l, r_{l+1}) &\text{ if $k=2l+1$},\\
(r_1,\vartheta_1,\ldots, r_l,\vartheta_l) & \text{if $k=2l$.}
\end{array}$$
 We will discuss the case $k=2l+1$ for the rest of the proof. For the case $k=2l$, the reader can just ignore the last coordinate. The advantage of these coordinates  is that the action of $g_n$ has a particularly easy form. In particular 
 the image $g_n (r_1,\vartheta_1,\ldots, r_l,\vartheta_l, r_{l+1})$ is given by 
$$\left(r_1\exp\left({\frac1{n^{l+1}}}\right),\vartheta_1+\frac{2\pi}{n},\ldots, r_l\exp\left({\frac1{n^{l+1}}}\right),\vartheta_l+\frac{2\pi}{n^l}, r_{l+1}\exp\left({\frac1{n^{l+1}}}\right)\right).$$

We claim that the region
$$R := \left\{(r_1,\vartheta_1,\ldots, r_l,\vartheta_l, r_{l+1}) \mid  \vartheta_i\in\left(-\pi-\frac{\pi}{n},-\pi+\frac{\pi}{n}\right), \; r_i\in\left(n-\frac12,n+\frac12\right)\right\}$$ 
is contained in a fundamental domain for the action of $\langle g_n = \rho_n(1) \rangle$ on $\bR^k$. In fact, suppose by contradiction that is not the case, it means we have two points $p, q \in R$ and $N \in \bZ$ such that $\rho_n(N)p = q$. Let $(r_1^p,\vartheta_1^p,\ldots, r_l^p,\vartheta_l^p, r_{l+1}^p)$ and $(r_1^q,\vartheta_1^q,\ldots, r_l^q,\vartheta_l^q, r_{l+1}^q)$ be the coordinates of $p$ and $q$ respectively. The calculation above shows that
$$
\begin{cases}
\displaystyle{r_i^p\exp\left({\frac{N}{n^{l+1}}}\right)}=r_i^q\\
\displaystyle{\vartheta_i^p+\frac{2\pi N}{n^i}}=\vartheta_i^q \end{cases}
$$

Since the angles should be in the range determined by the region $R$, 
we deduce from the  equation on $\vartheta_1$ that there is $d_1\in\bN$ such that $N = d_1 n$, the  equation on $\vartheta_2$ then implies that $d_1$ is also necessarily divisible by $n$, and thus there exists $d_2$ such that $N=d_2n^2$. We then conclude by induction using the other equations that $N=d_{l}n^l$. On the other hand, we can see that if $N=d_ln^l$, then it will violate the constrains on the radii determined by $R$. This proves that $R$ is contained in a fundamental domain for the action of $\langle g_n\rangle$.

Furthermore, for $n$ big enough, the region $R$ contains a ball $B$ of radius $\frac13$ around the origin. Note that the origin corresponds, in the polar  coordinates defined and used above, to the point $(n,-\pi,\ldots,n,-\pi,n)$. We can then choose disjoint open balls $B_+, B_-\subset B$, and fix an hyperbolic element $h\in\SO_0(1,k+1)$  with attractive (resp. repulsive) fixed point in $B_+$ (resp. $B_-$) and such that $h\cdot B_-^c\subset B_+$, where $B_\pm^c$ is the complement of $B_\pm$ in $\partial\HH^k_\bR\setminus\{\infty\}$. For every $n$ the pair $g_n:=\rho_n(1)$ and $h$ form a Schottky pair, and thus the associated representation $\varrho_n\co F_2 = \langle a, b \rangle \to\SO_0(1,k+1)$ defined by $\varrho(a) = g_n$ and $\varrho(b) = h$ is convex cocompact.
\end{proof} 

The proof in the complex hyperbolic case is very similar. 

\begin{Theorem}\label{t.Schotky_cpx}
There exists a sequence $\left(\varrho_n\co F_2\to \SU(1,k+1)\right)_{n \in \bN}$ of convex-cocompact representations such that the restriction $\varrho_n|_{\langle a\rangle}$ is the sequence of representations constructed in  Proposition  \ref{hyp_Jor2}.
\end{Theorem}

\begin{proof}
Let $\left(\rho_n\co\bZ\to \SU(1,k+1)\right)_{n \in \bN}$ denote the sequence of representations constructed in the proof of Proposition  \ref{hyp_Jor2}, and let $g_n = \rho_n (1)$. Recall that $g_n$ is a well positioned loxodromic with fixed points $\infty$ and $(\mathbf x_n; t) = (n, \ldots, n; 0)$, and its geometric data is given by 
$$\left(\mathbf{x}_n, t,  \frac{1}{n^{k+2}}, \theta_1=\frac{2\pi}{n}+\frac{2\pi}{n^{k+1}},\theta_2=\frac{2\pi}{n^2}, \ldots, \theta_{k-1}=\frac{2\pi}{n^{k-1}}, \theta_k=-\sum_{j =1}^{k-1}\frac{2\pi}{n^j}\right)$$

Given a point $p\in\Heis_k$, we denote by $(a(p),b(p))$ its standard coordinates (cfr. Section \ref{s.bdry}).
We will show that, for sufficiently big $n$, the product of a ball $B$ of radius $\frac13$ around the origin in $\bC^k$  and the interval $[-\frac13,\frac 13]$ is a subset of $\Heis_k$ that is contained in a fundamental region for the action of $\langle \rho_n(1) = g_n \rangle$ on $\partial \HH^k_{\mathbf C}$.

Assume by contradiction that there exists $p\in B$ and $N\in\bN$ such that both $p$ and $g_n^N(p)$ belong to $B$. Combining Proposition \ref{p.algdata2} and Proposition \ref{l.param}, we deduce that 
$$\begin{cases}
{\displaystyle a(g_n^N p)=\lambda_N e^{-i\phi_N} A^N (a(p)-{\bf x}_n)+ {\bf x}_n }\;\\
{\displaystyle b(g_n^N p)=\lambda_N^2b(p)+\Im\left(\lambda_N e^{-i\phi_N}\bar{\bf x}_n^TA^Na(p)-\lambda_N^2 \bar {\mathbf x}_n^T  a(p)- \lambda_N e^{-i\phi_N}(\bar{\bf x}_n^TA^N{\bf x}_n)\right)},\\
\end{cases}
$$
where $A$ is the rotational matrix associated to $g_n$, $\lambda_N={\displaystyle\frac{N}{n^{k+2}}}$, and $\phi_N={\displaystyle-\frac{\pi N}{n^{k+1}}}$.

We consider the equation on $a$ first. To simplify the notation, we consider again polar coordinates $(r_1(p),\vartheta_1(p),\ldots, r_k(p), \vartheta_k(p))$ on  $\bC^k$ centered at ${\bf x}_n$. Note that with respect to these coordinates, the origin has coordinates $(n,-\pi,\ldots,n,-\pi)$. Since the angle coordinates of any point in the ball $B$ satisfy  $\vartheta_i(p)\in\left(-\pi-\frac{\pi}{n},-\pi+\frac{\pi}{n}\right)$, and 
$$\vartheta_i(g_n^N p)=\begin{cases} 
{\displaystyle \vartheta_1 (p)+ \frac{2\pi N}{n}+\frac {3\pi N}{n^{k+1}}} & i=1\\
{\displaystyle \vartheta_i (p)+ \frac{2\pi N}{n^i}+\frac {\pi N}{n^{k+1}}}& i=2,\ldots, k-1\\
{\displaystyle \vartheta_k (p)-\sum_{j =1}^{k-1}\frac{2\pi N}{n^j}+\frac {\pi N}{n^{k+1}}} & i=k,
\end{cases}$$
we deduce, with the same argument as in the proof of Theorem \ref{t.Schotky_real}, that $N$ is necessarily a multiple of $n^{k-1}$. Furthermore, the equation on the radii
$$r_i(g_n^N p)=r_i(p)\exp\left(\frac{N}{n^{k+2}}\right)$$
implies that $N<n^{k+1}$. In fact, if $N\geq n^{k+1}$, we deduce from the previous equation that $r_i(g_n^N p)\geq r_i(p)\left(1+\frac 1n\right)$, which contradicts the assumption that both $p$ and $g_n^N p$ belong to $B$, since, for each point $q\in B$, it holds  $r_i(q)\in (n-\frac 13,n+\frac 13)$. The first angle equation, then, gives that $N=o(n^{k+1})$.

We can now consider the equation on $b$. Observe that, using the assumption $N = d_Nn^{k-1}$ and $d_N=o(n^2)$, we obtain 
$$\begin{array}{rl}
\Im(\lambda_N e^{-i\phi_N}\bar{\bf x}_n^TA^Na(p)-\lambda_N^2 \bar {\mathbf x}_n^T  a(p))&=\lambda_Nn\sum_{i=1}^k\Im ((e^{i(N\theta_i-\phi_N)}-\lambda_N)a_{i}(p)) \\
&= o(1)
\end{array}
$$

As a result the equation on $b$ can be simplified to
$$\begin{array}{ll}
b(g_n^Np)&={\displaystyle b(p)\lambda_N^2- n^2\lambda_N \left(\sum_{j=1}^k\sin\left(N(\theta_i-\phi) \right)\right)}+ o(1)\\
&={\displaystyle b(p)\exp\left(\frac{2d_N}{n^{3}}\right)- n^2\exp\left(\frac{d_N}{n^{3}}\right)
\left(\sin\left( \frac{3\pi d_N}{n^{2}}\right)+({k-1}) \sin\left( \frac{d_N\pi}{n^{2}}\right)+ o(1) \right)}\\\\
&={\displaystyle b(p)- 
({k+2}) {d_N\pi}+ o(1)},
\end{array}
$$
which is impossible, unless $d_N=0$, since both $g_n^Np$ and $p$ belong to $B$.
   
   So, for $n$ big enough, the region $R$ contains the product of a ball $B$ of radius $\frac13$ around the origin in $\bC^k$  and the interval $[-\frac13,\frac 13]$. We then conclude as in the case of the real hyperbolic space. 
\end{proof}

As a corollary we deduce extra algebraic information on the limit of the sequences considered in Propositions \ref{hyp_Jor} and \ref{hyp_Jor2}:
\begin{cor}\label{c.hyp_Jor}
The geometric limits of the sequences $(\rho_n^\bR:\bZ\to\SO_0(1,k+1))_{n\in\bN}$ and $(\rho_n^\bC:\bZ\to\SU(1,k+1))_{n\in\bN}$ in Propositions \ref{hyp_Jor}  and \ref{hyp_Jor2} are discrete and torsion free.
\end{cor}
\begin{proof}
It follows from Proposition \ref{prop.2.14} that the geometric limit $\Gamma_G^{F_2}$ of a sequence of convex cocompact representations of free groups is discrete and torsion free. The same is thus true for the limit $\Gamma_G$ of the  sequence $\rho_n(\bZ)$ in Proposition \ref{hyp_Jor} (resp.  \ref{hyp_Jor2}), since $\Gamma_G<\Gamma_G^{F_2}$.
\end{proof}
With this at hand we can conclude the proof of Theorem \ref{t.INTRO1}:
\begin{proof}[Proof of Theorem \ref{t.INTRO1}]
We know from Corollary \ref{c.hyp_Jor} that the limit of the sequences is a discrete, free abelian group, its rank is bigger or equal than $l+1$ by Proposition \ref{hyp_Jor} (resp. $k+1$ by Proposition \ref{hyp_Jor2}) and smaller or equal than $l+1$ by Proposition  \ref{p.thatsallR} (resp. $k+1$ because 
the highest rank of a discrete abelian subgroup of ${\sf G} = \SU(1, k + 1)$ is $k + 1$). The result follows. 
\end{proof}

\bibliographystyle{amsalpha}
\bibliography{adsbib}

\end{document}